\documentclass[oneside,11pt]{amsart}
\usepackage{amsmath}
\usepackage{amssymb}
\usepackage{graphicx}
\usepackage{caption}
\usepackage{subcaption}
\usepackage{pinlabel}
\usepackage{tikz}

\newtheorem{propappend}{Proposition}
\newtheorem{lemappend}[propappend]{Lemma}
\newtheorem{thm}{Theorem}[section]
\newtheorem{prop}[thm]{Proposition}
\newtheorem{lem}[thm]{Lemma}
\newtheorem{cor}[thm]{Corollary}

\theoremstyle{definition}
\newtheorem{defn}[thm]{Definition}
\newtheorem{rem}[thm]{Remark}

\newtheorem{ques}[thm]{Question}

\newcommand{\abs}[1]{\lvert{#1}\rvert}

\renewcommand{\bar}[1]{\overline{#1}}

\newcommand{\presentation}[2]{\langle\, {#1} \mid {#2} \,\rangle}

\newcommand{\bigset}[2]{ \bigl\{ \, {#1} \bigm| {#2} \, \bigr\} }
\renewcommand{\emptyset}{\varnothing}

\newcommand{\field}[1]{\mathbb{#1}}
\newcommand{\Z}{\field{Z}}

\newcommand{\R}{\field{R}}

\DeclareMathOperator{\supp}{supp}
\DeclareMathOperator{\St}{St}
\DeclareMathOperator{\csupp}{csupp}

\DeclareMathOperator{\CAT}{CAT}
\DeclareMathOperator{\Lk}{Lk}

\DeclareMathOperator{\diam}{diam}

\begin{document}

\title[Malnormality and join-free subgroups]{Malnormality and join-free subgroups in right-angled Coxeter groups}

\author{Hung Cong Tran}
\address{Department of Mathematics\\
 The University of Georgia\\
1023 D. W. Brooks Drive\\
Athens, GA 30605\\
USA}
\email{hung.tran@uga.edu}

\date{\today}

\begin{abstract}
In this paper, we prove that all finitely generated malnormal subgroups of one-ended right-angled Coxeter groups are strongly quasiconvex and they are in particular quasiconvex when the ambient groups are hyperbolic. The key idea is to prove all infinite proper malnormal subgroups of one-ended right-angled Coxeter groups are join-free and then prove the strong quasiconvexity and the virtual freeness of these subgroups. We also study the subgroup divergence of join-free subgroups in right-angled Coxeter groups and compare them with the analogous subgroups in right-angled Artin groups. We characterize almost malnormal parabolic subgroups in terms of their defining graphs and also recognize them as strongly quasiconvex subgroups by the recent work of Genevois and Russell-Spriano-Tran. Finally, we discuss some results on hyperbolically embedded subgroups in right-angled Coxeter groups.
\end{abstract}

\maketitle

\section{Introduction}

It is well-known that quasiconvex subgroups of hyperbolic groups have finite height (see \cite {MR1389776}). The \emph{height} of a subgroup $H$ in a group $G$ is the smallest number $n$ with the property that for any $(n+1)$ distinct left cosets $g_1H, g_2H\cdots, g_{n+1}H$ the intersection $\bigcap g_iH$ is always finite. Swarup asked if the converse is true:

\begin{ques}[\cite{Bestvina}]
\label{q1} 
Let $G$ be a hyperbolic group and $H$ a finitely generated subgroup. If $H$ has finite height, is $H$ quasiconvex?
\end{ques}

Gitik stated that the problem is open even when $H$ is malnormal in $G$. A subgroup $H$ of a group $G$ is \emph{malnormal} if $gHg^{-1}\cap H$ is trivial for each $g$ not in $H$. Wise and Agol also suggested one could attempt to answer it for hyperbolic \emph{virtually special groups} (i.e. groups that virtually embed into some right-angled Artin group), but even that seems tricky.

\begin{ques}
\label{q2} 
Is a malnormal finitely generated subgroup of a hyperbolic (virtually special) group quasiconvex?
\end{ques}
 
We observe that the above two questions can be extended to the analogous subgroups of arbitrary finitely generated groups. In \cite{MR3426695}, Durham-Taylor introduced a strong notion of quasiconvexity in finitely generated groups, called \emph{stability}, which is preserved under quasi-isometry, and which agrees with quasiconvexity when the ambient group is hyperbolic. However, a stable subgroup of a finitely generated group is always hyperbolic regardless of the geometry of the ambient group (see \cite{MR3426695}). Thus, the geometry of a stable subgroup does not completely reflect that of the ambient group. Therefore, the author~\cite{Tran2017} and Genevois~\cite{Genevois2017} independently introduced another concept of quasiconvexity, called \emph{strong quasiconvexity}, which is strong enough to be preserved under quasi-isometry and relaxed enough to capture the geometry of ambient groups.

\begin{defn}
Let $G$ be a finitely generated group and $H$ a subgroup of $G$. We say $H$ is \emph{strongly quasiconvex} in $G$ if for every $K \geq 1,C \geq 0$ there is some $M = M(K,C)$ such that every $(K,C)$--quasi–geodesic in $G$ with endpoints on $H$ is contained in the $M$--neighborhood of $H$.
\end{defn}

In \cite{Tran2017}, the author also characterized stable subgroups as hyperbolic strongly quasiconvex subgroups. He also proved that strongly quasiconvex subgroups of finitely generated groups also have finite height (see Theorem 1.2 in \cite{Tran2017}). Therefore, it is reasonable to extend Question~\ref{q1} and Question~\ref{q2} to strongly quasiconvex subgroups of finitely generated groups.

\begin{ques}
\label{q3} 
Let $G$ be a finitely generated (virtually special) group and $H$ a finitely generated subgroup. If $H$ has finite height (or $H$ is malnormal), is $H$ strongly quasiconvex?
\end{ques}

We note that the work of the author in \cite{Tran2017} implicitly gave the positive answer to Question~\ref{q3} for one-ended right-angled Artin groups. More precisely, if a finitely generated subgroup of a one-ended right-angled Artin group has finite height, then it is strongly quasiconvex. In the Appendix, we give an explicit proof of this fact. Moreover, we provide necessary conditions for finite height subgroups of groups satisfying certain conditions (see Proposition~\ref{fhimpliessq} and Lemma~\ref{simplecase}) and we hope this may be useful for someone who wants to attack Question~\ref{q3} for different group collections in future.

\subsection{Malnormality in right-angled Coxeter groups}
The positive answer for Question~\ref{q3} for one-ended right-angled Artin groups motivate us to work on the richer collection of groups, called right-angled Coxeter groups. For each finite simplicial graph $\Gamma$ the associated \emph{right-angled Coxeter group} $G_\Gamma$ has generating set $S$ equal to the vertices of $\Gamma$, relations $s^2=1$ for each $s$ in $S$ and relations $st = ts$ whenever $s$ and $t$ are adjacent vertices. \textbf{In this paper, we assume all graphs that define some right-angled Coxeter group are finite and simplicial.} In contrast to right-angled Artin groups, the collection of right-angled Coxeter groups contains numerous hyperbolic groups, relatively hyperbolic groups, and thick groups of arbitrary orders. Right-angled Coxeter groups also provide a rich source of cubical groups and any results on this collection can shed light on extensions to all cubical groups. By some recent work on characterizing strongly quasiconvex parabolic subgroups of right-angled Coxeter groups (see Proposition 4.9 in \cite{Genevois2017} or Theorem 7.5 in \cite{JDT2018}) we can prove easily that a finite height parabolic subgroup of a right-angled Coxeter group is strongly quasiconvex (see Proposition~\ref{fhpsracgs}). However, it seems difficult to extend this result to arbitrary finite height subgroups of right-angled Coxeter groups. Therefore, we focus our work on their malnormal subgroups and we obtain a positive answer.

\begin{thm}
\label{motivation1}
Let $\Gamma$ be a connected graph and $H$ a finitely generated malnormal subgroup of the right-angled Coxeter group $G_\Gamma$. Then $H$ is strongly quasiconvex. Moreover, if $H$ is a proper subgroup, then $H$ is virtually free (therefore, $H$ is also stable).
\end{thm} 

Since right-angled Coxeter groups is a large class of virtually special groups (see \cite{MR2646113}), Theorem~\ref{motivation1} sheds light on the positive answer to Question~\ref{q2} and the ability to extend our result to the non-hyperbolic case via the concept of strongly quasiconvex subgroups (see Question~\ref{q3}). In Theorem~\ref{motivation1}, if $H$ is finite or $H=G_\Gamma$, then $H$ is strongly quasiconvex clearly. Otherwise, $H$ is an infinite proper subgroup and this implies that $\Gamma$ is not a join by Remark~\ref{rmotivation}. Therefore, the key idea for the proof of Theorem~\ref{motivation1} is the following theorem.

\begin{thm}
\label{keystep1}
Let $\Gamma$ be a non-join connected graph and $H$ an infinite proper malnormal subgroup of the right-angled Coxeter group $G_\Gamma$. Then $H$ is a join-free subgroup (i.e. $H$ is infinite and none of infinite order elements in $H$ are conjugate into a join subgroup).
\end{thm}

Theorem~\ref{keystep1} is the main motivation for studying join-free subgroups of right-angled Coxeter groups with connected non-join defining graphs. We note that all finitely generated join-free subgroups of right-angled Coxeter groups with connected non-join defining graphs are stable and virtually free by Propositions \ref{vfree}, \ref{prop2} and \ref{prop3}. 
We also refer the reader to Section~\ref{stt1} for the proof of Theorem \ref{keystep1}. In general, we still do not know whether an arbitrary almost malnormal (or even finite height) subgroup of a one-ended right-angled Coxeter group is strongly quasiconvex although the positive answer is already confirmed for parabolic subgroups as we discussed at the beginning. Note that a subgroup $H$ of a group $G$ is \emph{almost malnormal} if $gHg^{-1}\cap H$ is finite for each $g$ not in $H$. It is clear that an infinite almost malnormal subgroup has height exactly $1$.

We end this section with a characterization of almost malnormal parabolic subgroups in right-angled Coxeter groups which does not seem to be recorded in the literature. 

\begin{prop}
\label{almost}
Let $\Gamma$ be a finite simplicial graph and $\Lambda$ be an induced subgraph of $\Gamma$. Then a parabolic subgroup $H$ of right-angled Coxeter group $G_\Gamma$ induced by $\Lambda$ is almost malnormal if and only if no vertex of $\Gamma-\Lambda$ commutes to two non-adjacent vertices of $\Lambda$.
\end{prop}

Using Proposition 4.9 in \cite{Genevois2017} or Theorem 7.5 in \cite{JDT2018} we can easily see that all almost malnormal parabolic subgroups of right-angled Coxeter groups are strongly quasiconvex but as discussed above strong quasiconvexity even also holds for all finite height parabolic subgroups of right-angled Coxeter groups. However, the above proposition will later help us characterize hyperbolically embedded parabolic subgroups of right-angled Coxeter groups.

\subsection{Hyperbolically embedded subgroups in right-angled Coxeter groups} 

Hyperbolically embedded subgroups are generalizations of peripheral subgroups in relatively hyperbolic groups (see \cite{DGO}) and are a key component of studying acylindrically hyperbolic groups, a large class of groups exhibiting hyperbolic-like behavior (see \cite{MR3430352}). Work of Dahmani-Guirardel-Osin \cite{DGO} and Sisto \cite{MR3519976} showed that if a finite collection of subgroups $\{H_i\}$ is hyperbolically embedded in a finitely generated group $G$, then $\{H_i\}$ is an almost malnormal collection and each $H_i$ is strongly quasiconvex. The converse of this statement is true for groups acting geometrically on $\CAT(0)$ cube complexes (see Theorem 6.31 in \cite{Genevois2017}) and hierarchically hyperbolic groups (see Theorem H in \cite{JDT2018}) which both includes right-angled Coxeter groups. We note that a collection $\mathcal{H}$ of subgroups of $G$ is \emph{malnormal} (resp. \emph{almost malnormal}) if for each $H, H'\in\mathcal{H}$ and $g\in G$ we have $H\cap gH'g^{-1}\neq\{e\}$ (resp. $\abs{H\cap gH'g^{-1}}=\infty\}$) implies $H=H'$ and $g\in H$. Therefore, the following is a corollary of Theorem~\ref{motivation1}. 

\begin{cor}
\label{cmotivation1}
Let $\Gamma$ be a connected graph and $\mathcal{H}$ a finite collection of finitely generated subgroups of the right-angled Coxeter group $G_\Gamma$. If $\mathcal{H}$ is malnormal, then $\mathcal{H}$ is hyperbolically embedded. In addition, the converse is also true if all subgroups in $\mathcal{H}$ are torsion free.
\end{cor} 



We note that all proper hyperbolically embedded subgroups in the above corollary are virtually free (and therefore hyperbolic). Using the work of Caprace~\cite{MR2665193,MR3450952} we can construct a non-trivial relatively hyperbolic right-angled Coxeter group with non-hyperbolic peripheral subgroups. Therefore, the peripheral subgroups are clearly non-hyperbolic hyperbolically embedded subgroups. Outside the relatively hyperbolic setting, one may expect that all proper hyperbolically embedded subgroups of right-angled Coxeter groups are hyperbolic. However, this is not true. Combining the characterization of parabolic strongly quasiconvex subgroups (see Proposition 4.9 in \cite{Genevois2017} or Theorem 7.5 in \cite{JDT2018}) and the characterization of almost malnormal collections of parabolic subgroups (see Corollary~\ref{cmotivation4}) in right-angled Coxeter groups we obtain a characterization of hyperbolically embedded collections of parabolic subgroups in this group collection. 

\begin{prop}
\label{cmotivation3}
Let $\Gamma$ be a simplicial finite graph and $$\mathcal{H}=\{g_1G_{\Lambda_1}g_1^{-1}, g_2G_{\Lambda_2}g_2^{-1}, \cdots, g_nG_{\Lambda_n}g_n^{-1}\}$$ a collection of parabolic subgroups of the right-angled Coxeter group $G_\Gamma$. Then $\mathcal{H}$ is hyperbolically embedded (i.e. $\mathcal{H}$ is almost malnormal) in $G_\Gamma$ if and only if the following holds:
\begin{enumerate}
\item For each $\Lambda_i$ no vertex outside $\Lambda_i$ commutes to non-adjacent vertices of $\Lambda_i$; and
\item $\Lambda_i\cap \Lambda_j$ is empty or a clique for each $i\neq j$. 
\end{enumerate}
\end{prop}

We now use the above proposition to construct an example of proper non-hyperbolic hyperbolically embedded subgroup of a non-relatively hyperbolic right-angled Coxeter group. Let $\Gamma$ be the graph $\Gamma_3$ in Figure 7 of \cite{JDT2018} and let $\Lambda$ be the red $4$--cycle as in the figure. Then it is clear that $G_\Lambda$ is a virtually $\field{Z}^2$ hyperbolically embedded proper subgroup of the $\mathcal{CFS}$ right-angled Coxeter group $G_\Gamma$. We note that the $\mathcal{CFS}$ condition on defining graphs was used in Dani-Thomas \cite{MR3314816} and Levcovitz \cite{IL} to characterize right-angled Coxeter groups with quadratic divergence. Since the divergence of a one-ended relatively hyperbolic groups is always exponential (see \cite{Sisto}), $\mathcal{CFS}$ right-angled Coxeter groups are never relatively hyperbolic. Actually, we can use the proof of Corollary G in \cite{JDT2018} to prove that every right-angled Coxeter group is a hyperbolically embedded subgroup of some $\mathcal{CFS}$ right-angled Coxeter group. 

\subsection{Geometric embedding properties of join-free subgroups and their generalization}

We note that join-free subgroups were also defined analogously for right-angled Artin groups in Koberda-Mangahas-Taylor~\cite{KMT} under the name \emph{purely loxodromic subgroups}. They also proved that the such groups are strongly quasiconvex and free. The reader can see later that we mostly follow their strategy for the proof of the strong quasiconvexity and the virtual freeness of our groups (see Section~\ref{copy1} and Section~\ref{copy2}). However, we show the embedding properties of our subgroups in right-angled Coxeter groups are more diverse than the ones of the analogous subgroups in right-angled Artin groups.

\begin{thm}
\label{nea}
For each $d\geq 2$ there is a right-angled Coxeter group $G_d$ such that for each $2\leq m\leq d$ the group $G_d$ contains a join-free subgroup $H_d^m$ which is isomorphic to the group $F=\presentation{a,b,c}{a^2=b^2=c^2=e}$ and whose subgroup divergence in $G_d$ is a polynomial of degree $m$.
\end{thm} 

Subgroup divergence was introduced by the author with the name \emph{lower relative divergence} in \cite{MR3361149} to study geometric embedding properties of a subgroup inside a finitely generated group. We note that the subgroup divergence of a join-free subgroup in a one-ended right-angled Artin group is always quadratic (see Corollary 1.17 in \cite{Tran2017}). Therefore, the geometric embedding properties of join-free subgroups in right-angled Coxeter groups are more plentiful. However, we also prove the quadratic subgroup divergence holds for certain class of right-angled Coxeter groups. 

\begin{thm}
Let $\Gamma$ be a non-join connected $\mathcal{CFS}$ graph and $H$ a finitely generated join-free subgroup of the right-angled Coxeter group $G_{\Gamma}$. Then the subgroup divergence of $H$ in $G_{\Gamma}$ is exactly quadratic. 
\end{thm}

As observed above, join-free subgroups are proved to be useful to study the malnormality in right-angled Coxeter groups. However, if one only cares about the coarse geometry of subgroups that are similar to the one of join-free subgroups, the concept of join-free subgroups seems to be quite restrictive because it requires that the defining graph of the ambient group to be not a join. Therefore, we proposed a concept of almost join-free subgroups for all right-angled Coxeter groups whose defining graphs are not a join of two subgraphs of diameters at least $2$. More precisely, if the ambient graph $\Gamma$ is not a join of two subgraphs with diameters at least $2$, then we can write $\Gamma=\Gamma_1*K$ where $K$ is a (possibly empty) clique and $\Gamma_1$ is a non-join graph. In this case, $G_{\Gamma_1}$ is a finite index subgroup of $G_\Gamma$ and we can extend the concept of join-free subgroups to subgroups in $G_\Gamma$ as follows. An infinite subgroups $H$ of $G_\Gamma$ is \emph{almost join-free} if $H\cap G_{\Gamma_1}$ is a join-free subgroup of $G_{\Gamma_1}$. It is clear that if $\Gamma$ is not a join, an almost join-free subgroup of $G_\Gamma$ is a truly join-free subgroup in $G_\Gamma$. 


\section{Preliminaries}
\subsection{Coarse geometry}
\label{sub1}

We first review the concepts of quasi-isometric embedding, quasi-isometry, quasi-geodesics, geodesics, undistorted subgroups, strongly quasiconvex subgroups, stable subgroups, and subgroup divergence.

\begin{defn} For metric spaces $(X,d_X)$ and $(Y,d_Y)$ be two metric spaces and constants $K \geq 1$ and $L \geq 0$, a map $f\!:X \to Y$ is a \emph{$(K, L)$--quasi-isometric embedding} if for all $x_1, x_2 \in X$, 

A \emph{quasi-isometric embedding} is simply a $(K,L)$--quasi-isometric embedding for some $K,L$. When a quasi-isometric embedding $f\!:X \to Y$ has the additional property that every point in $Y$ is within a bounded distance from the image $f(X)$, we say $f$ is a \emph{quasi-isometry} and $X$ and $Y$ are \emph{quasi-isometric}.

Where $X$ is a subinterval $I$ of $\R$ or $\Z$, we call a $(K, L)$--quasi-isometric embedding $f\!:I \to Y$ a \emph{$(K, L)$--quasi-geodesic}. If $K = 1$ and $L = 0$, then $f\!:I \to Y$ is a \emph{geodesic}.
\end{defn} 

\begin{defn}
Let $G$ be a finitely generated group and $H$ a finitely generated subgroup of $G$. We say $H$ is \emph{undistorted} in $G$ if the inclusion map of subgroup $H$ into the group $G$ is a quasi-isometric embedding (this is independent of the word metrics on $H$ and $G$). We say $H$ is \emph{strongly quasiconvex} in $G$ if for every $K \geq 1,C \geq 0$ there is some $M = M(K,C)$ such that every $(K,C)$--quasi–geodesic in $G$ with endpoints on $H$ is contained in the $M$--neighborhood of $H$. We say $H$ is \emph{stable} in $G$ if $H$ is undistorted in $G$, and for any $K \geq 1$ and $L\geq 0$ there is an $M = M(K,L) \geq 0$ such that any pair of $(K,L)$--quasi-geodesics in $G$ with common endpoints in $H$ have Hausdorff distance no greater than $M$.
\end{defn}

In \cite{Tran2017} the author proved that a subgroup is stable if and only if it is strongly quasiconvex and hyperbolic.

Before we define the concept of subgroup divergence, we need to introduce the notions of domination and equivalence which are the tools to measure the subgroup divergence.
\begin{defn}
Let $\mathcal{M}$ be the collection of all functions from $[0,\infty)$ to $[0,\infty]$. Let $f$ and $g$ be arbitrary elements of $\mathcal{M}$. \emph{The function $f$ is dominated by the function $g$}, denoted \emph{$f\preceq g$}, if there are positive constants $A$, $B$, $C$ and $D$ such that $f(x)\leq Ag(Bx)+Cx$ for all $x>D$. Two function $f$ and $g$ are \emph{equivalent}, denoted \emph{$f\sim g$}, if $f\preceq g$ and $g\preceq f$.

\end{defn}

\begin{rem}
A function $f$ in $\mathcal{M}$ is \emph{linear, quadratic or exponential...} if $f$ is respectively equivalent to any polynomial with degree one, two or any function of the form $a^{bx+c}$, where $a>1, b>0$.
\end{rem}

\begin{defn}
Let $\{\delta^n_{\rho}\}$ and $\{\delta'^n_{\rho}\}$ be two families of functions of $\mathcal{M}$, indexed over $\rho \in (0,1]$ and positive integers $n\geq 2$. \emph{The family $\{\delta^n_{\rho}\}$ is dominated by the family $\{\delta'^n_{\rho}\}$}, denoted \emph{$\{\delta^n_{\rho}\}\preceq \{\delta'^n_{\rho}\}$}, if there exists constant $L\in (0,1]$ and a positive integer $M$ such that $\delta^n_{L\rho}\preceq \delta'^{Mn}_{\rho}$. Two families $\{\delta^n_{\rho}\}$ and $\{\delta'^n_{\rho}\}$ are \emph{equivalent}, denoted \emph{$\{\delta^n_{\rho}\}\sim \{\delta'^n_{\rho}\}$}, if $\{\delta^n_{\rho}\}\preceq \{\delta'^n_{\rho}\}$ and $\{\delta'^n_{\rho}\}\preceq \{\delta^n_{\rho}\}$.
\end{defn}

\begin{rem}
A family $\{\delta^n_{\rho}\}$ is dominated by (or dominates) a function $f$ in $\mathcal{M}$ if $\{\delta^n_{\rho}\}$ is dominated by (or dominates) the family $\{\delta'^n_{\rho}\}$ where $\delta'^n_{\rho}=f$ for all $\rho$ and $n$. The equivalence between a family $\{\delta^n_{\rho}\}$ and a function $f$ in $\mathcal{M}$ can be defined similarly. Thus, a family $\{\delta^n_{\rho}\}$ is linear, quadratic, exponential, etc if $\{\delta^n_{\rho}\}$ is equivalent to the function $f$ where $f$ is linear, quadratic, exponential, etc.
\end{rem}

\begin{defn}
Let $X$ be a geodesic space and $A$ a subspace of $X$. Let $r$ be any positive number.
\begin{enumerate}
\item $N_r(A)=\bigset{x \in X}{d_X(x, A)<r}$
\item $\partial N_r(A)=\bigset{x \in X}{d_X(x, A)=r}$ 
\item $C_r(A)=X-N_r(A)$.
\item Let $d_{r,A}$ be the induced length metric on the complement of the $r$--neighborhood of $A$ in $X$. If the subspace $A$ is clear from context, we can use the notation $d_r$ instead of using $d_{r,A}$. 
\end{enumerate}
\end{defn}

\begin{defn}
Let $(X,A)$ be a pair of geodesic spaces. For each $\rho \in (0,1]$ and positive integer $n\geq 2$, we define a functions $\sigma^n_{\rho}\!:[0, \infty)\to [0, \infty]$ as follows: 

For each positive $r$, if there is no pair of $x_1, x_2 \in \partial N_r(A)$ such that $d_r(x_1, x_2)<\infty$ and $d(x_1,x_2)\geq nr$, we define $\sigma^n_{\rho}(r)=\infty$. Otherwise, we define $\sigma^n_{\rho}(r)=\inf d_{\rho r}(x_1,x_2)$ where the infimum is taken over all $x_1, x_2 \in \partial N_r(A)$ such that $d_r(x_1, x_2)<\infty$ and $d(x_1,x_2)\geq nr$. 


The family of functions $\{\sigma^n_{\rho}\}$ is \emph{the subspace divergence} of $A$ in $X$, denoted $div(X,A)$.
\end{defn}

We now define the subgroup divergence of a subgroup in a finitely generated group.
\begin{defn} 
Let $G$ be a finitely generated group and $H$ its subgroup. We define \emph{the subgroup divergence} of $H$ in $G$, denoted \emph{$div(G,H)$}, to be the subspace divergence of $H$ in the Cayley graph $\Gamma(G,S)$ for some finite generating set $S$. 

\end{defn}

\begin{rem}
The concept of subgroup divergence was introduced by the author with the name \emph{lower relative divergence} in \cite{MR3361149}. The subgroup divergence is a pair quasi-isometry invariant concept (see Proposition 4.9 in \cite{MR3361149}). This implies that the subgroup divergence of a subgroup on a finitely generated group does not depend on the choice of finite generating sets of the whole group. 
\end{rem}

\subsection{Geometry and algebra of right-angled Coxeter groups}
\label{sss}

In this section, we review the concepts of right-angled Coxeter groups, special subgroups, parabolic subgroups, star subgroups, join subgroup, Davis complexes, and some basic algebraic and geometric properties of right-angled Coxeter groups. 

\begin{defn}
Given a finite, simplicial graph $\Gamma$, the associated \emph{right-angled Coxeter group} $G_\Gamma$ has generating set $S$ the vertices of $\Gamma$, and relations $s^2 = 1$ for all $s$ in $S$ and $st = ts$ whenever $s$ and $t$ are adjacent vertices.

Let $S_1$ be a subset of $S$. The subgroup of $G_\Gamma$ generated by $S_1$ is a right-angled Coxeter group $G_{\Gamma_1}$, where $\Gamma_1$ is the induced subgraph of $\Gamma$ with vertex set $S_1$ (i.e. $\Gamma_1$ is the union of all edges of $\Gamma$ with both endpoints in $S_1$). The subgroup $G_{\Gamma_1}$ is called a \emph{special subgroup} of $G_\Gamma$. Any of its conjugates is called a \emph{parabolic subgroup} of $G_\Gamma$.

A \emph{reduced word} for a group element $g$ in $G_{\Gamma}$ is a minimal length word in the free group $F(S)$ representing $g$. It is proved in \cite{MR1314099} that if $w = v_1v_2\cdots v_p$ is not reduced, then there exists $1 \leq i < j \leq p$ such that $v_i = v_j$ and $v_i$ is adjacent to each of the vertices $v_{i+1},\cdots, v_{j-1}$ (the \emph{Deletion Condition}). Moreover, it is also proved in \cite{MR1314099} that if two reduced words $w$, $w'$ define the same element of $G_\Gamma$, then $w$ can be transformed into $w'$ by a finite number of letter swapping operations (the \emph{Transpose Condition}).

Let $w$ be any word in the vertex generators. We say that $v\in S$ is in the \emph{support} of $w$, written $v\in \supp(w)$, if $v$ occurs as a letter in $w$. For $g \in G_\Gamma$ and $w$ a reduced word representing $g$, we define the \emph{support} of $g$, $\supp(g)$, to be $\supp(w)$. We define the \emph{cyclically support} of $g$, $\csupp(g)$, to be the intersection of all sets $\supp(wgw^{-1})$, where each $w$ is a group element in $G_\Gamma$. It follows from Transpose Condition that $\supp(g)$ and $\csupp(g)$ are well-defined. We say that $u$ is \emph{cyclically reduced} if $\csupp(u)=\supp(u)$. It is also well know that each $g\in G_\Gamma$ has a unique reduced expression $wuw^{-1}$ with $u$ cyclically reduced and therefore $\csupp(g)=\supp(u)$. 
\end{defn}

\begin{defn}
Let $\Gamma_1$ and $\Gamma_2$ be two non-empty graphs, the \emph{join} of $\Gamma_1$ and $\Gamma_2$ is a graph obtained by connecting every vertex of $\Gamma_1$ to every vertex of $\Gamma_2$ by an edge.

Let $J$ be an induced subgraph of $\Gamma$ which decomposes as a join. We call $G_J$ a \emph{join subgroup} of $G_{\Gamma}$. A reduced word $w$ in $G_\Gamma$ is called a \emph{join word} if $w$ represents element in some join subgroup. If $\beta$ is a subword of $w$, we will say that $\beta$ is a \emph{join subword} of $w$ when $\beta$ is itself a join word.

For a vertex $v$ of the graph $\Gamma$ let $\Lk(v)$ denote the subgraph of $\Gamma$ induced by the vertices adjacent to $v$ called the \emph{link} of $v$ and let $\St(v)$ denote the subgraph spanned by $v$ and $\Lk(v)$ called the \emph{star} of $v$. The special subgroup $G_{\St(v)}$ is a \emph{star subgroup} of $G_\Gamma$. Note that a star of a vertex is always a join, but the converse is generally not true. A reduced word $w$ in $G_\Gamma$ is called a \emph{star word} if $w$ represents element in some star subgroup. If $\beta$ is a subword of $w$, we will say that $\beta$ is a \emph{star subword} of $w$ when $\beta$ is itself a star word. Note that a star word is always a join word, but the converse is generally not true.
\end{defn}

\begin{defn}
Given a finite, simplicial graph $\Gamma$, the associated \emph{Davis complex} $\Sigma_\Gamma$ is a cube complex constructed as follows. For every $k$--clique, $T \subset \Gamma$, the special subgroup $G_T$ is isomorphic to the direct product of $k$ copies of $Z_2$. Hence, the Cayley graph of $G_T$ is isomorphic to the 1--skeleton of a $k$--cube. The Davis complex $\Sigma_\Gamma$ has 1--skeleton the Cayley graph of $G_\Gamma$, where edges are given unit length. Additionally, for each $k$--clique, $T \subset \Gamma$, and coset $gG_T$, we glue a unit $k$--cube to $gG_T \subset\Sigma_\Gamma$. The Davis complex $\Sigma_\Gamma$ is a $\CAT(0)$ space and the group $G_\Gamma$ acts properly and cocompactly on the Davis complex $\Sigma_\Gamma$ (see \cite{MR2360474}).
\end{defn}

The idea for the following lemma comes from Lemma 3.1 in \cite{MR2874959}. Moreover, the proof of the following lemma is almost identical to the proof of that lemma. Therefore, we here just copy the proof Lemma 3.1 in \cite{MR2874959} with slight changes that are suitable to the case of RACGs.

\begin{lem} 
\label{lbc}
Let $H_1 = g_1H_v$ and $H_2 = g_2H_w$. Then
\begin{enumerate}
\item $H_1$ intersects $H_2$ if and only if $v$, $w$ commute and $g_1^{-1}g_2 \in G_{\St(v)}G_{\St(w)}$.
\item There is a hyperplane $H_3$ intersecting both $H_1$ and $H_2$ if and only if there is $u$ in $\St(v) \cap \St(w)$ such that $g_1^{-1}g_2 \in G_{\St(v)}G_{\St(u)}G_{\St(w)}$.
\end{enumerate}
\end{lem}

\section{Join-free subgroups and malnormality in right-angled Coxeter groups}
\label{stt1}

In this section, we define the concepts of join-free subgroups and star-free subgroups in right-angled Coxeter groups. We study the connections among parabolic subgroups, star-free subgroups, and join-free subgroups. We also give a proof of the theorem that an infinite proper malnormal subgroup of a right-angled Coxeter group with connected defining graph is always join-free. Finally we characterize almost malnormal parabolic subgroups and almost malnormal collections of parabolic subgroups in right-angled Coxeter groups in terms of their defining subgraphs.

\begin{defn}
Let $\Gamma$ be a simplical finite graph. An infinite subgroup $H$ of the right-angled Coxeter group $G_\Gamma$ is \emph{join-free} if none of its infinite order elements are conjugate into a join subgroup. An infinite subgroup $H$ of $G_\Gamma$ is \emph{star-free} if none of its infinite order elements are conjugate into a star subgroup. 
\end{defn}

\begin{rem}
It is clear from the definition that if the ambient graph $\Gamma$ is a join (resp. a star), then the right-angled Coxeter group $G_\Gamma$ contains no join-free subgroup (resp. no star-free subgroup). Therefore, whenever we assume the right-angled Coxeter group $G_\Gamma$ contains a join-free subgroup (resp. a star-free subgroup) the ambient graph $\Gamma$ is understood implicitly to be not a join (resp. not a star).

It is clear that a join-free subgroup of $G_\Gamma$ is star-free, but the converse is false. For example, we can chose $\Gamma$ as a square labeled cyclically by the vertices $a, b, c, d$. Then \[G_\Gamma=\langle a,c\rangle \times \langle b,d\rangle \cong D_\infty \times D_\infty.\]
Since $\Gamma$ is a join graph, $G_\Gamma$ has no join-free subgroup. However, any cyclic group generated by cyclically reduced word with full support is star-free. In particular, the cyclic subgroup $\langle abcd \rangle$ is a star-free subgroup.

\end{rem}

Now we can connect parabolic subgroups, star-free subgroups, and join-free subgroups.

\begin{prop}
\label{exoplgips}
Let $\Gamma$ be a non-join connected graph. Let $H$ be a conjugate of a special subgroup induced by a subset $S_1$ of vertex set of $\Gamma$. Then the following are equivalent: 
\begin{enumerate}
\item $S_1$ contains at least two non-adjacent vertices and the distance in $\Gamma$ between any two elements of $S_1$ is different from 2. 
\item $H$ is join-free.
\item $H$ is star-free.
\end{enumerate}
A subgroup $H$ satisfying some (all) above condition is virtually a free group.
\end{prop}
 
\begin{proof}
Since any join-free subgroup is star-free, then we only need to prove (1) implies (2), and (3) implies (1). Without the loss of generality we can assume that $H$ is a special subgroup. We first prove that (3) implies (1). In fact, if vertices in $S_1$ are pairwise adjacent, then $H$ is a finite subgroup and then $H$ is not star-free. If $H$ has two vertices $u$ and $v$ with distance 2 in $\Gamma$, then $h=uv$ is an infinite order of $H$ which belongs to some star subgroup. Therefore, $H$ is not a star-free subgroups in this case.

We now prove that (1) implies (2). Assume that $H$ is not join-free. Then there is an infinite order element $h$ in $H$ that is conjugate to a join subgroup. Then $\csupp(h)$ is a subset of the vertex set of some induced join subgraph $\Gamma_1$. Since $h$ is an infinite order element of the special group generated by $S_1$, $\csupp(h)$ is a subset of $S_1$ and there are two vertices $v_1$ and $v_2$ in $\csupp(h)$ that are not adjacent in $\Gamma$. Since two non-adjacent vertices $v_1$ and $v_2$ both lie in the join subgraph $\Gamma_1$, the distance in $\Gamma$ between $v_1$ and $v_2$ is exactly 2. This is a contradiction. Therefore, $H$ is a join-free subgroup.

We observe that if $S_1$ contains at least two non-adjacent vertices and the distance in $\Gamma$ between any two elements of $S_1$ is different from 2, then the subgraph induced by $S_1$ is disconnected and each component is a single point or a clique. Therefore, $H$ is a free product of more than one finite subgroups. This implies that $H$ is a virtually free subgroup.
\end{proof} 

By the above proposition, parabolic join-free subgroups are always virtually free subgroups. We remark that any infinite subgroup of a join-free subgroup is also join-free. Therefore, we conclude that any infinite subgroup which is conjugate into a join-free special subgroup is also virtually free join-free subgroup. In general, we will show that a join-free subgroup is not necessarily conjugate into a join-free special subgroup. However, we will prove later that a join-free subgroups is always virtually free even when it is not conjugate into a join-free special subgroup.

\begin{figure}
\begin{tikzpicture}[scale=1]

\draw (0,0) node[circle,fill,inner sep=1pt, color=black](1){} -- (2,0) node[circle,fill,inner sep=1pt, color=black](1){}-- (4,0) node[circle,fill,inner sep=1pt, color=black](1){}-- (6,0) node[circle,fill,inner sep=1pt, color=black](1){} -- (3,2) node[circle,fill,inner sep=1pt, color=black](1){}-- (0,0) node[circle,fill,inner sep=1pt, color=black](1){}-- (2,-1) node[circle,fill,inner sep=1pt, color=black](1){}-- (4,0) node[circle,fill,inner sep=1pt, color=black](1){};

\draw (2,0) node[circle,fill,inner sep=1pt, color=black](1){} -- (4,1) node[circle,fill,inner sep=1pt, color=black](1){}-- (6,0) node[circle,fill,inner sep=1pt, color=black](1){};

\draw (2,0) node[circle,fill,inner sep=1pt, color=black](1){} -- (2,-1) node[circle,fill,inner sep=1pt, color=black](1){};

\draw (4,0) node[circle,fill,inner sep=1pt, color=black](1){} -- (4,1) node[circle,fill,inner sep=1pt, color=black](1){};

\node at (0,-0.25) {$a$};
\node at (2,0.25) {$b$};
\node at (4,-0.25) {$c$};
\node at (6,0.25) {$d$};
\node at (3,2.25) {$t$};
\node at (2,-1.25) {$d_1$};
\node at (3.75,1.15) {$a_1$};

\end{tikzpicture}

\caption{}
\label{aa}
\end{figure}

We now come up with an example of a join-free subgroup which is not conjugate into a join-free special subgroup. Let $\Gamma$ be a graph in Figure \ref{aa}. Then we observe that the distance between any two non-adjacent vertices in $\Gamma$ is exactly two. Therefore, the group $G_\Gamma$ does not contains any join-free parabolic subgroups by Proposition \ref{exoplgips}. Let $x=(aa_1)(dd_1)(aa_1)$, $y=(dd_1)(aa_1)(dd_1)$, and $H$ a subgroup generated by $x$ and $y$. Then $H$ is a free subgroup of rank two and $H$ is also a join-free subgroup (see the following proposition). 

\begin{prop}
\label{ppss}
Let $\Gamma$ be a graph in Figure \ref{aa} and $H$ a subgroup generated by $x=(aa_1)(dd_1)(aa_1)$ and $y=(dd_1)(aa_1)(dd_1)$. Then $H$ is a free subgroup of rank two and $H$ is also a join-free subgroup.
\end{prop}

\begin{proof}
Let $S$ be the vertex set of $\Gamma$ and $T=\{x,y,x^{-1},y^{-1}\}$. Let $w=u_1u_2\cdots u_n$ be an arbitrary freely reduced word in $T$ and $\bar{w}$ be the word obtained from $w$ by replacing $x$, $x^{-1}$, $y$, $y^{-1}$ by their corresponding subwords in $G_\Gamma$. We remark that $w$ and $\bar{w}$ both represent the same element in $H$. We will prove that $\bar{w}$ is a reduced word in $G_\Gamma$. 

Since $w$ is a freely reduced word in $T$, then subword of two consecutive elements $u_iu_{i+1}$ in $w$ must lie in \[\{xx, x^{-1}x^{-1}, yy, y^{-1}y^{-1}, xy, y^{-1}x^{-1}, yx, x^{-1}y^{-1}, x^{-1}y, y^{-1}x, xy^{-1}, yx^{-1}\}.\] By using the Deletion Condition, we can check that any subwords in $\bar{w}$ that replaces two consecutive elements $u_iu_{i+1}$ in $w$ is reduced. Assume for the contradiction that $\bar{w}$ is not a reduced word in $G_\Gamma$. Then using the Deletion Condition, there exists $1 \leq \ell < k \leq 6n$ such that the $\ell^{th}$ element $v_{\ell}$ and the $k^{th}$ element $v_k$ in $\bar{w}$ are labelled by the same generator in $S$ and $v_{\ell}$ commutes with all elements between $v_{\ell}$ and $v_k$. We can assume further that no element of $\bar{w}$ between $v_{\ell}$ and $v_k$ is labelled by the same generator as $v_{\ell}$ and $v_k$. Also, any subword of $\bar{w}$ that replaces $x$, $y$, $x^{-1}$, $y^{-1}$ has the same support as $\bar{w}$. Therefore, $v_{\ell}$ and $v_k$ must lie in the subword that replaces two consecutive elements $u_iu_{i+1}$ in $w$. This implies that the subword that replaces two consecutive elements $u_iu_{i+1}$ in $w$ is not reduced. This is a contradiction. Therefore, $\bar{w}$ is a reduced word in $G_\Gamma$. This implies that $H$ is a free subgroup of rank $2$ and $\abs{h}_S=6\abs{h}_T$ for each element $h$ in $H$. This fact also implies that if $h$ is cyclically reduced in $(H,T)$, then $h$ is also cyclically reduced in $(G_\Gamma,S)$. 

We now assume for the contradiction that $H$ is not a join-free subgroup. Then there is a nontrivial element $h$ that is conjugate into a join subgroup. We can assume that $h$ is cyclically reduced in $(H,T)$. Therefore, $h$ is also cyclically reduced in $(G_\Gamma,S)$ and $h$ lies in a join subgroup. Therefore, the support $\supp(h)=\{a,a_1,d,d_1\}$ must lie in the vertex set of some join subgraph $\Gamma'=\Gamma_1*\Gamma_2$. Since the subgraph of $\Gamma$ induced by $\supp(h)$ is not a join, then $\supp(h)=\{a,a_1,d,d_1\}$ must lie entirely in $\Gamma_1$ or $\Gamma_2$ (say $\Gamma_1$). Therefore, $\supp(h)=\{a,a_1,d,d_1\}$ lies entirely in the star of some vertex in $\Gamma_2$. We can check easily that this is a contradiction. Therefore, $H$ is a join-free subgroup. 
\end{proof}

We will prove later that all join-free subgroups in RACGs are stable. However, the converse is not true. For example, a cyclic subgroup $H$ of a right-angled Coxeter group $G_\Gamma$ generated by a rank-one isometry $g$ is stable but $H$ is not a join-free subgroup when $g$ is conjugate into a star subgroup. We can also construct a non-virtually cyclic stable subgroup which is not join-free as follows. Let $\Gamma$ be a connected graph which has no separating clique and no embedded cycles of length four. We assume also that $\Gamma$ contains an embedded cycle $C$ of length more than four. Then the right-angled Coxeter group $G_\Gamma$ is a one-ended hyperbolic group (see Theorem 8.7.2 and Corollary 12.6.3 in \cite{MR2360474}) and the special subgroup $G_C$ is a non-virtually cyclic quasiconvex subgroup of $G_\Gamma$. Therefore, $G_C$ is a non-virtually cyclic stable subgroup. It is obvious that the vertex set of $C$ does not satisfy conditions in Proposition \ref{exoplgips}. Therefore, $G_C$ is not a join-free subgroup.

We now prove that an infinite proper malnormal subgroup of a right-angled Coxeter group with non-join connected defining graph is always join-free.

\begin{proof}[Proof of Theorem~\ref{keystep1}] 

We first prove that for each vertex $s$ of $\Gamma$ and each group element $g$ in $G_\Gamma$ the group element $gsg^{-1}$ never be an element of $H$. Assume for a contradiction that there is a vertex $s_0$ of $\Gamma$ and a group element $g_0$ in $G_\Gamma$ such that $g_0s_0g_0^{-1}$ is a group element of $H$. Therefore, $s_0$ is a group element of group $K=g_0^{-1}Hg_0$. We note that $K$ is also a malnormal subgroup. Let $s$ be an arbitrary adjacent vertex of $s_0$. Then we see that $sKs\cap K$ contains the non-identity element $s_0$. Therefore, $s$ must be also a group element of $K$. Since $\Gamma$ is connected, all vertices of $\Gamma$ must be group elements of $K$. This implies that $K$ (also $H$) is the ambient group $G_\Gamma$ which is a contradiction. Therefore, for each vertex $s$ of $\Gamma$ and each group element $g$ in $G_\Gamma$ the group element $gsg^{-1}$ never be an element of $H$. 

We now assume for a contradiction that $H$ is not a join-free subgroup. Then there is an infinite order element $h$ in $H$ such that $h$ belongs to some parabolic subgroup $gG_\Lambda g^{-1}$ where $\Lambda$ is a join of two other subgraphs $\Lambda_1$ and $\Lambda_2$. We note that $H\cap gG_\Lambda g^{-1}$ is also malnormal in $gG_\Lambda g^{-1}$. If both subgraphs $\Lambda_1$ and $\Lambda_2$ have diameter at least 2, then $H\cap gG_\Lambda g^{-1}=gG_\Lambda g^{-1}$ by Lemma~\ref{simplecase}. In particular, for each vertex $s$ of $\Lambda$ the group element $gsg^{-1}$ belongs to $H$ which is a contradiction. We now consider the case either $\Lambda_1$ or $\Lambda_2$ (say $\Lambda_1$) is consists of a single vertex or has diameter $1$. Let $s$ be an arbitrary vertex of $\Lambda_1$. Then the group elements $g'=gsg^{-1}$ commutes to all elements in $gG_\Lambda g^{-1}$. In particular, $g'$ commutes to $h$ and therefore $g'$ is a group element of $H$ which is also a contradiction. Therefore, $H$ must be a join-free subgroup. 
\end{proof}


\begin{rem}
\label{rmotivation}
We remark that if $\Gamma$ is a join graph, then right-angled Coxeter group $G_\Gamma$ does not contain any infinite proper malnormal subgroup. In fact if $\Gamma$ is a join of two subgraphs of diameters at least $2$, then $G_\Gamma$ contains no infinite proper malnormal subgroup by Lemma~\ref{simplecase}. Otherwise, $\Gamma$ is a join of a subgraph of diameter at least $2$ and a non-empty clique. In this case there is a vertex $v$ of $\Gamma$ that commutes to all groups elements of $G_\Gamma$. By the proof of Theorem~\ref{keystep1}, if $H$ is an infinite proper malnormal subgroup of $G_\Gamma$, then the vertex $v$ never be an element of $H$. Moreover, since $v$ commutes to all groups elements of $G_\Gamma$, we have $vHv=H$ which is a contradiction. Therefore, if $\Gamma$ is a join graph, then the right-angled Coxeter group $G_\Gamma$ does not contain any infinite proper malnormal subgroup. 
\end{rem}

We now characterize almost malnormal parabolic subgroups in right-angled Coxeter groups.

\begin{proof}[Proof of Proposition~\ref{almost}]
We note that if a subgroup is almost malnormal, then all its conjugates are also almost malnormal. Therefore, we can assume that $H=G_\Lambda$ is the special subgroup induced by $\Lambda$. We first assume $G_\Lambda$ is almost malnormal and we will prove that no vertex of $\Gamma-\Lambda$ commutes to two non-adjacent vertices of $\Lambda$. Assume for a contradiction that there is a vertex $u$ in $\Gamma-\Lambda$ that commutes with two non-adjacent vertices $v_1$ and $v_2$ of $\Lambda$. Then subgroup $uG_\Lambda u^{-1}\cap G_\Lambda$ contains an infinite order group element $v_1v_2$. Since subgroup $G_\Lambda$ is almost malnormal, $u$ must be a group element of $G_\Lambda$ and this implies that $u$ is a vertex of $\Lambda$ which is a contradiction. Thus, no vertex of $\Gamma-\Lambda$ commutes to two non-adjacent vertices of $\Lambda$.

We now assume that no vertex of $\Gamma-\Lambda$ commutes to two non-adjacent vertices of $\Lambda$ and we will prove that $G_\Lambda$ is almost malnormal. Assume for a contradiction that $G_\Lambda$ is not almost malnormal. Then there is a group element $g$ not in $G_\Lambda$ such that $gG_\Lambda g^{-1}\cap G_\Lambda$ is infinite. We can choose such element $g$ such that $\abs{g}_S$ is minimal where $S$ is the vertex set of $\Gamma$. Let $w_0$ be a reduced word in $S$ that represents $g$. Then the reverse word $\bar{w}_0$ of $w_0$ is a reduced word that represents $g^{-1}$. Since $g$ is not an element of $G_\Lambda$, some element in $w_0$ must be a vertex of $\Gamma-\Lambda$.

Since $gG_\Lambda g^{-1}\cap G_\Lambda$ is infinite, there is an infinite order element $h$ in $G_\Lambda$ such that $ghg^{-1}$ is also an element in $G_\Lambda$. Let $w_1$ be a reduced word in $S$ that represents $h$. Then all elements of $w_1$ are vertices of $\Lambda$ and there are at least two of them which are non-adjacent vertices of $\Lambda$. Therefore, the concatenation $w=w_0w_1\bar{w}_0$ represents the group element $ghg^{-1}$ in $G_\Lambda$. We can write $w=v_1v_2\cdots v_p$. 

Since $w$ contains a vertex not in $\Lambda$, then $w$ is not reduced. Then there exists $1 \leq i < j \leq p$ such that $v_i = v_j$ and $v_i$ is adjacent to each of the vertices $v_{i+1},\cdots, v_{j-1}$. Since $w_0$, $w_1$, and $\bar{w}_0$ are all reduced, $v_i$ and $v_j$ can not lie in the same block in $w$. We first assume that $v_i$ lies in $w_0$ and $v_j$ lies in $w_1$. Then $v_i$ is a vertex in $\Lambda$ and we have $g=g'v_i$ where $\abs{g'}_S=\abs{g}_S-1$. Therefore, $g'G_\Lambda g'^{-1}\cap G_\Lambda=gG_\Lambda g^{-1}\cap G_\Lambda$ which contradicts to the choice of $g$. By an analogous argument we also get the same contradiction if we assume $v_i$ lies in $w_1$ and $v_j$ lies in $\bar{w}_0$. Therefore, $v_i$ must lie in $w_0$ and $v_j$ must lie in $\bar{w}_0$. Moreover, $v_i$ is not a vertex of $\Lambda$ by an analogous argument. Since $w_1$ contains at least two non-adjacent vertices of $\Lambda$, the vertex $v_i$ must commute to both these vertices which is a contradiction. Therefore, $G_\Lambda$ is almost malnormal.
\end{proof}

Combining Proposition~\ref{almost} above and Proposition 3.4 in \cite{MR3365774} we obtain the following corollary which characterizes almost malnormal collections of parabolic subgroups in right-angled Coxeter groups.

\begin{cor}
\label{cmotivation4}
Let $\Gamma$ be a simplicial finite graph and $$\mathcal{H}=\{g_1G_{\Lambda_1}g_1^{-1}, g_2G_{\Lambda_2}g_2^{-1}, \cdots, g_nG_{\Lambda_n}g_n^{-1}\}$$ a collection of parabolic subgroups of the right-angled Coxeter group $G_\Gamma$. Then $\mathcal{H}$ is an almost malnormal collection in $G_\Gamma$ if and only if the following holds:
\begin{enumerate}
\item For each $\Lambda_i$ no vertex outside $\Lambda_i$ commutes to non-adjacent vertices of $\Lambda_i$; and
\item $\Lambda_i\cap \Lambda_j$ is empty or a clique for each $i\neq j$. 
\end{enumerate}
\end{cor}

\section{Dual van Kampen diagrams for right-angled Coxeter groups}
\label{copy1}
In this section, we construct dual van Kampen diagrams for right-angled Coxeter groups which are almost identical to dual van Kampen diagrams for right-angled Artin groups constructed in \cite{MR2443098}. In \cite{KMT}, Koberda-Mangahas-Taylor used dual van Kampen diagrams for right-angled Artin groups to study the geometry of their join-free subgroups (under the name purely loxodromic subgroups) and star-free subgroups. In Sections \ref{copy1} and \ref{copy2} of this article, we will follow the same strategy as in \cite{KMT} to study the geometry of join-free subgroups and star-free subgroups in right-angled Coxeter groups. 
\subsection{Formal definition}
We now develop dual van Kampen diagrams for right-angled Coxeter groups. The key ingredient for constructing such diagrams for RACGs which are similar to ones for RAAGs is the similarity between Davis complexes and universal covers of Salvetti complexes.

Let $\Gamma$ be a graph with the vertex set $S$. Let $w$ be a word representing the trivial element in $G_\Gamma$. A \emph{dual van Kampen diagram} $\Delta$ for $w$ in $G_\Gamma$ is an oriented disk $D$ together with a collection $\mathcal{A}$, properly embedded arcs in general position, satisfying the following:

\begin{enumerate}
\item Each arc of $\mathcal{A}$ is labeled by an element of $S$. Moreover, if two arcs of $\mathcal{A}$ intersect then the generators corresponding to their labels are adjacent in $\Gamma$.
\item With its induced orientation, $\partial D$ represents a cyclic conjugate of the word $w$ in the following manner: there is a point $*\in \partial D$ such that $w$ is spelled by starting at $*$, traversing $\partial D$ according to its orientation, and recording the labels of the arcs of $\mathcal{A}$ it encounters
\end{enumerate}

We think of the boundary of $D$ as subdivided into edges and labeled
according to the word $w$. In this way, each arc of $\mathcal{A}$ corresponds to two letters of $w$ which are represented by edges on the boundary of $D$. While not required by the definition, it is convenient to restrict our attention to \emph{tight} dual van Kampen diagrams, in which arcs of $\mathcal{A}$ intersect at most once.

In comparison with dual van Kampen diagrams for RAAGs, the only difference from dual van Kampen diagrams for RACGs is we do not need a direction equipped on each embedded arc of $\mathcal{A}$ and each edge of $\partial D$. The key reason for this difference is each edge of universal covers of Salvetti complexes is equipped with a direction while each edge of Davis complexes is not.

We now show the way to construct the dual van Kampen diagram for an identity word $w$ in a right-angled Coxeter groups. Let $\tilde{\Delta}\subset S^2$ be a (standard) van Kampen diagram for $w$, with respect to a standard presentation of $G_\Gamma$. Consider $\tilde{\Delta}^*$, the dual of $\tilde{\Delta}$ in $S^2$, and name the vertex which is dual to the face $S^2-\tilde{\Delta}$ as $v_\infty$. Then for a sufficiently small ball $B(v_\infty)$ around $v_\infty$, $\tilde{\Delta}^*-B(v_\infty)$ can be considered as a dual van Kampen diagram with a suitable choice of the labeling map. Therefore a dual van Kampen diagram exists for any word $w$ representing the trivial element in $G_\Gamma$. Conversely, a van Kampen diagram $\tilde{\Delta}$ for a word can be obtained from a dual van Kampen diagram $\Delta$ by considering the dual complex again. So, the existence of a dual van Kampen diagram for a word $w$ implies that $w$ represents the trivial element in $G_\Gamma$.
\subsection{Surgery and subwords}

Let $\Gamma$ be a graph with the vertex set $S$. Starting with a dual van Kampen diagram $\Delta$ with a disk $D$ and collection $\mathcal{A}$ of embedded arcs in $D$ for an identity word $w$. Suppose that $\gamma$ is a properly embedded arc in $\Delta$ which is either an arc of $\mathcal{A}$ or transverse to the arcs of $\mathcal{A}$. Traversing $\gamma$ in some direction and recording the labels of those arcs of $\mathcal{A}$ that cross $\gamma$ spells a word $y$ in the standard generators. We say the word $y$ is obtained from \emph{label reading} along $\gamma$ with the chosen direction.

In particular, starting with a subword $w'$ of $w$, any oriented arc 
of $D$ which begins at the initial vertex of $w'$ and ends at the terminal vertex of $w'$ produces a word $y$ via label reading such that $w'= y$ in $G_\Gamma$. To see this, we observe that the arc $\gamma$ cuts the disk $D$ into two disks $D_1$ and $D_2$, one of which (say $D_1$) determines the homotopy (and sequence of moves) to transform the word $w'$ into $y$. In other word, the disk $D_1$ along with arcs from $\mathcal{A}$ forms a dual van Kampen diagram for the word $w'\bar{y}$, and we say that this diagram is obtained via \emph{surgery} on $\Delta$. It is is straightforward that if the arc $\gamma$ is labelled by a vertex $v$ in $S$, then $w'$ represents an element in the star subgroup $G_{\St(v)}$. You can see the following lemma for a precise statement.

\begin{lem}
\label{3.1}
Suppose an arc of $\mathcal{A}$ in a dual van Kampen diagram $\Delta$ for the identity word $w$ cuts off the subword $w'$, i.e., $w \equiv svw'vt$, where $s$, w', and $t$ are subwords and $v$ is the letter at the ends of the arc. Then $w'$ represents a group element in the star subgroup $G_{\St(v)}$.
\end{lem} 

If a subword in a dual van Kampen diagram has the property that no two arcs emanating from it intersect, this subword is \emph{combed} in
the dual van Kampen diagram. We remark that this such type of subword was also defined for dual van Kampen diagrams for RAAGs in \cite{KMT} and it played an important role to study some certain types of subgroups of RAAGs. In Sections \ref{copy1} and \ref{copy2} of this article, we are following the same strategy in \cite{KMT} to study subgroups of RACGs. Therefore, the property of being combed will be important in these sections. 

\begin{lem}
\label{3.2}
Suppose $w$ is a word representing the identity and $b$ is a subword of $w$, so $w$ is the concatenation of words $a$, $b$, and $c$. Let $\Delta$ be a dual van Kampen diagram for $w$.

Then there exists a word $b'$ obtained by re-arranging the letters in $b$, such that $b' = b$ and there exists a dual van Kampen diagram $\Delta'$ for $ab'c$ in which $b'$ is combed, arcs emanating from $b'$ have the same endpoint in the boundary subword $ca$ as their counterpart in $b$, and arcs that both begin and end in $ca$ are unchanged in $\Delta'$.

Furthermore, there exists a word $b''$ obtained by deleting letters in $b'$, such that $b'' = b$ and there exists a dual van Kampen diagram $\Delta''$ for $ab''c$ which is precisely $\Delta'$ without the arcs corresponding to the deleted letters.
\end{lem}

The above lemma is identical to Lemma 3.2 in \cite{KMT} for RAAGs. Moreover, we observe that the proof of Lemma 3.2 in \cite{KMT} can be applied to prove the above lemma. Therefore, the reader can see the proof of Lemma 3.2 in \cite{KMT} to obtain the proof of the above lemma.

\subsection{Reducing diagrams}

In subsection 3.5 in \cite{KMT}, Koberda-Mangahas-Taylor introduce reducing diagrams and some related concepts to study words in RAAGs as well as paths in universal covers of Salvetti complexes. We observed that these concepts are also well-defined for the case of RACGs and they can also help us studying words in RACGs as well as paths in Davis complexes. Therefore, we just copy most of subsection 3.5 in \cite{KMT} and the reader can verify easily that these materials fit well for the case of RACGs. 

Let $h$ be a word in the vertex generators of $G_\Gamma$, which is not assumed to be reduced in any sense. Let $w$ denote a reduced word in the vertex generators which represents the same group element as $h$ does. Then, the word $h\bar{w}$ represents the identity in $G_\Gamma$ and so it is the boundary of some dual van Kampen diagram $\Delta$. (Here $\bar{w}$ denotes the inverse of the word $w$.) In this way, the boundary of $\Delta$ consist of two words $h$ and $\bar{w}$. We sometimes refer to a dual van Kampen diagram constructed in this way as a \emph{reducing diagram} as it represents a particular way of reducing $h$ to the reduced word $w$. For such dual van Kampen diagrams, $\partial D$ is divided into two subarcs (each a union of edges) corresponding to the words $h$ and $w$, we call these subarcs the $h$ and $w$ subarcs, respectively.

Suppose that $\Delta$ is a dual van Kampen diagram that reduces $h$ to the reduced word $w$. Since $w$ is already a reduced word, no arc of $\mathcal{A}$ can have both its endpoints on the $w$ subarc of $\partial D$. Otherwise, one could surger the diagram to produce a word equivalent to $w$ with fewer letters. Hence, each arc of $\mathcal{A}$ either has both its endpoints on the subarc of $\partial D$ corresponding to $h$, or it has one endpoint in each subarc of $\partial D$. In the former case, we call the arc (and the letters of $h$ corresponding to its endpoints) \emph{noncontributing} since these letters do not contribute to the reduced word $w$. Otherwise, the arc is called \emph{contributing} (as is
the letter of $h$ corresponding the endpoint contained in the $h$ subarc of $\partial D$). If the word $h$ is partitioned into a product of subwords $abc$, then the \emph{contribution of the subword $b$ to $w$} is the set of letters in $b$ which are contributing. We remark that whether a letter of $h$ is contributing or not is a property of the fixed dual van Kampen diagram that reduces $h$ to $w$.

\section{The geometry of subgroups of right-angled Coxeter groups}
\label{copy2}
In this section, we prove the undistortedness and freeness for star-free subgroups and the stability for join-free subgroups in right-angled Coxeter groups. Our work follows the same strategy in \cite{KMT} for proving the analogous properties in right-angled Artin groups but is based on the dual van Kampen diagrams for right-angled Coxeter groups developed in the previous section. Throughout this section, we assume that $\Gamma$ is connected and not a join.
 
\subsection{The geometry of star-free subgroups}
Recall that a nontrivial subgroup $H$ of $G_\Gamma$ is \emph{star-free} if each infinite order element in $H$ is not conjugate into a star subgroup. We now assume that $H$ is a finitely generated star-free subgroup of $G_\Gamma$ with a finite generating set $T$. Therefore, each element $h\in H$ can be expressed as a geodesic word in $H$, that is, $h=h_1h_2\cdots h_n$ such that $h_i\in T$ and $n$ is minimal. We use a dual van Kampen diagram with boundary word $(h_1h_2\cdots h_n)h^{-1}$, where $h$ and each $h_i$ are written as reduced words in $G_\Gamma$. In other words, we concatenate the reduced word representatives for the $h_i$ to obtain a word representing $h = h_1\cdots h_n$ and consider a reducing diagram for this word. With our choices fixed, we call such a reducing diagram for $h$ simply a dual van Kampen diagram for $h \in H$.

The following lemma is identical to the Lemma 4.1 in \cite{KMT} for RAAGs. Moreover, their proofs are almost identical except there is some small extra step at the end of the proof of the following lemma.

\begin{lem}
\label{4.1}
Suppose $H$ is a finitely generated, star-free subgroup of $G_\Gamma$. There exists $D = D(H)$ with the following property: If in a dual van Kampen diagram for $h\in H$, a letter in $h_i$ is connected to a letter in $h_j$ ($i<j$), then $j-i < D$.
\end{lem}

\begin{proof}
Suppose in a dual van Kampen diagram for $h \in H$, a letter $g$ in $h_i$ is connected to another letter $g$ in $h_j$. By Lemma \ref{3.1}, $h_i\cdots hj = \sigma M \tau$, where $M$ is in the star of $g$, and $\sigma$, $\tau$ are a prefix of $h_i$ and suffix of $h_j$ respectively. Therefore, if the lemma is false, there a sequence of reduced-in-$H$ words \[h^{(t)}_{i(t)}\cdots h^{(t)}_{j(t)}=\sigma_t M_t \tau_t\]
as above, with $j(t)-i(t)$ strictly increasing. Because $\Gamma$ is finite and $H$ is finitely generated, we may pass to a subsequence so that the $M_t$ are in the star of the same generator $v$, and furthermore we have constant $\sigma_t = \sigma$ and $\tau_t = \tau$, while $M_t\neq M_s$ for $s\neq t$. Therefore, for each $t\geq 2$, element \[k_t=\bigl(h^{(t)}_{i(t)}\cdots h^{(t)}_{j(t)}\bigr)\bigl(h^{(1)}_{i(1)}\cdots h^{(1)}_{j(1)}\bigr)^{-1}=\sigma M_t M_1^{-1} \sigma^{-1}\] is nontrivial element in the subgroup $H\cap \sigma G_{\St(v)}\sigma^{-1}$. Moreover, $k_t\neq k_s$ for any $2\leq t<s$.

Assume that $k_{t_0}$ is infinite order for some $t_0\geq 2$. Then, $H$ is not a star-free subgroup which is a contradiction. We now assume that the order of all $k_t$ are two. Since $k_t\neq k_s$ for any $2\leq t<s$, we can choose two different elements $k_{t_1}$ and $k_{t_2}$ which do not commute. Therefore, the order of the group element $k_{t_1}k_{t_2}$ is not two. This implies that $k_{t_1}k_{t_2}$ is an infinite order element in the subgroup $H\cap \sigma G_{\St(v)}\sigma^{-1}$. Then, $H$ is not a star-free subgroup which is a contradiction.
\end{proof}

The following lemma is identical to the Lemma 4.2 in \cite{KMT} for RAAGs. Moreover, the proof of the below lemma almost follows the same line argument as in the proof of Lemma 4.2 in \cite{KMT}. Here we only need to replace Lemmas 3.2, 4.1 in the proof of Lemma 4.2 in \cite{KMT} by Lemmas \ref{3.2}, \ref{4.1} of this paper respectively to obtain the proof of the following lemma.

\begin{lem}
\label{4.2}
Suppose $H$ is a finitely generated, star-free subgroup of $G_\Gamma$ and $D$ is a constant as in Lemma \ref{4.1}. Let $h_i\cdots h_j$ be a subword of $h = h_1 \cdots h_n$ reduced in $H$ as above. Then the element $h_i \cdots h_j \in G_\Gamma$ may be written as a concatenation of three words $\sigma W \tau$, where the letters occurring in $\sigma$ are a subset of the letters occurring in $h_{i-D}\cdots h_{i-1}$ when $i>D$, and in $h_1\cdots h_{i-1}$ otherwise; the letters occurring in $\tau$ are a subset of the letters occurring in $h_{j+1}\cdots h_{j+D}$ when $j \leq n-D$ and in $h_{j+1}\cdots h_n$ otherwise; and the letters occurring in $W$ are exactly the letters occurring in $h_i \cdots h_j$ which survive in the word $h$ after it is reduced in $G_\Gamma$.
\end{lem}

The following lemma is a key lemma to prove that the star-free subgroup $H$ is undistorted in $G_\Gamma$. This lemma is identical to the Lemma 4.3 in \cite{KMT} for RAAGs and its proof again almost follows the same line argument as in the proof of Lemma 4.3 in \cite{KMT}. Here we only need to replace Lemma 4.2 in the proof of Lemma 4.3 in \cite{KMT} by Lemma \ref{4.2} of this paper to obtain the proof of the following lemma.

\begin{lem}
\label{4.3}
Given $H$ a finitely generated, star-free subgroup of $G_\Gamma$, there exists $K = K(H)$ such that, if $h_i \cdots hj$ is a subword of a reduced word for $h$ in $H$ which contributes nothing to the reduced word for $h$ in $G_\Gamma$, then $j - i < K$.
\end{lem}

The following proposition is a direct result of Lemma \ref{4.3}.
\begin{prop}
\label{4.4}
Finitely generated star-free subgroups are undistorted.
\end{prop}

The proof of the following proposition is almost identical to the proof of Theorem 53 in \cite{MR3192368}. We recall a proof with a slight modification for the convenience of the reader.

\begin{prop}
\label{vfree}
Star-free subgroups are virtually free.
\end{prop}

\begin{proof}

We first assume that $H$ is torsion free. We will prove that $H$ is a free subgroup by induction on the number of vertices of $\Gamma$. Since $H$ is a torsion free star-free subgroup, for each vertex $v$ of $\Gamma$ and $g$ in $G_\Gamma$ \[H\cap gG_{\St(v)}g^{-1}=\{1\}.\]

For the base case $\Gamma=v$, $G_\Gamma=\Z_2$ and $H=\{1\}$. Therefore, the result in this case is obvious. For the inductive step, choose a vertex $v$ of $\Gamma$ and let $\Gamma_v$ be the induced subgraph of $\Gamma$ generated by all vertices of $\Gamma$ except $v$. We observe that $G_\Gamma=G_{\St(v)}*_{G_{\Lk(v)}} G_{\Gamma_v}$. By standard Bass-Serre Theory, we see that $H$ acts on the corresponding Bass-Serre tree with trivial edge stabilizer. Therefore, there exists a (possibly infinite) collection of subgroups $\{H_i\}$ with each $H_i$ conjugate to $G_{\Gamma_v}$ in $G_\Gamma$ such that $H$ is a free product of subgroups $H_i$ with possibly an additional free
factor. Since $H_i$ is conjugate into $G_{\Gamma_v}$ and $\Gamma_v$ has fewer vertices than $\Gamma$, we see that $H$ is free by induction.

We now assume that $H$ is not torsion free. Let $G_1$ be a finite-index torsion free subgroup of $G_\Gamma$ and $H_1=G_1\cap H$ (see \cite{MR1709956} for a construction of group $G_1$). Then, $H_1$ is a torsion free star-free subgroup of $H$ with a finite index in $H$. Also, $H_1$ is a free subgroup by the above argument. This implies that $H$ is a virtually free subgroup.
\end{proof}

\subsection{Geometric embedding properties of join-free subgroups}
Assume the graph $\Gamma$ is a non-join connected graph. A nontrivial subgroup $H$ of $G_\Gamma$ is \emph{$N$--join-busting} if, for any reduced word $w$ representing $h$ in $H$, and any join subword $\beta\leq w$, the length of $\beta$ is bounded above by $N$.

By using almost the same line argument as in Section 5 in \cite{KMT}, we obtained the Proposition \ref{prop2} as below. We remark that the Proposition \ref{prop2} is identical to the Theorem 5.2 in \cite{KMT}. However, we need to use van Kampen diagrams for RACGs instead of van Kampen for RAAGs in \cite{KMT}. We also use Lemmas \ref{3.2}, \ref{4.1}, and \ref{4.2} of this paper instead of Lemmas 3.2, 4.1, and 4.2 in \cite{KMT} respectively.

\begin{prop}
\label{prop2}
Let $\Gamma$ be a non-join connected graph and $H$ a finitely generated join-free subgroup of the right-angled Coxeter group $G_\Gamma$. There exists an $N = N(H)$ such that $H$ is $N$--join-busting.
\end{prop}

In the Proposition \ref{prop3} as below, we prove the stability of $N$--join-busting subgroups in RACGs. This proposition is identical to Corollary 6.2 in \cite{KMT}. The proof of Proposition \ref{prop3} follows almost the same line argument as in Section 6 in \cite{KMT}. However, we need to use van Kampen diagrams for RACGs instead of van Kampen for RAAGs in and we use Proposition \ref{4.4} in this paper instead of Proposition 4.4 in \cite{KMT}. 

\begin{prop}
\label{prop3}
Let $\Gamma$ be a non-join connected graph and $H$ a finitely generated join-free subgroup of the right-angled Coxeter group $G_\Gamma$. If $H$ is $N$--join-busting for some $N$, then $H$ is stable in $G_\Gamma$. 
\end{prop}

\section{Subgroup divergence of join-free subgroups in right-angled Coxeter groups}

In this section, we study the subgroup divergence of join-free subgroups in right-angled Coxeter groups. We prove that the subgroup divergence of join-free subgroups in right-angled Coxeter groups can be polynomials of arbitrary degrees while it must be exactly quadratic in $\mathcal{CFS}$ right-angled Coxeter groups.

\subsection{Subgroup divergence in $\mathcal{CFS}$ right-angled Coxeter groups}
We first define the concept of $\mathcal{CFS}$ graphs.

\begin{defn}
Let $\Gamma$ be a non-join graph. We define the associated \emph{four-cycle} graph $\Gamma^4$ as follows. The vertices of $\Gamma^4$ are the induced loops of length four (i.e. four-cycles) in $\Gamma$. Two vertices of $\Gamma^4$ are connected by an edge if the corresponding four-cycles in $\Gamma$ share a pair of non-adjacent vertices. Given a subgraph $K$ of $\Gamma^4$, we define the \emph{support} of $K$ to be the collection of vertices of $\Gamma$ (i.e. generators of $G_\Gamma$) that appear in the four-cycles in $\Gamma$ corresponding to the vertices of $K$. Graph $\Gamma$ is said to be \emph{$\mathcal{CFS}$} if there exists a component of $\Gamma^4$ whose support is the entire vertex set of $\Gamma$.
\end{defn}

The following two lemmas contribute to the proof of the quadratic upper bound for the subgroup divergence of join-free subgroups in $\mathcal{CFS}$ right-angled Coxeter groups. 

\begin{lem}
\label{ltran}
Let $\Gamma$ be a non-join connected graph with the vertex set $S$. Let $H$ be a finitely generated join-free subgroup of $G_{\Gamma}$. There is a positive number $K$ such that the following holds. Let $g$ be an element in $G_\Gamma$ and $(s_1,t_1)$, $(s_2,t_2)$ two pairs of non-adjacent vertices in a four-cycle of $\Gamma$. Let $u_1=s_1t_1$ and $u_2=s_2t_2$. Then \[d_S(gu_1^{i}u_2^{j},H)\geq \frac{\abs{i}+\abs{j}}{K}-\abs{g}_S-1.\] 
\end{lem}

\begin{proof}
By Proposition \ref{prop2}, there is a positive integer $N$ such that for any reduced word $w$ representing $h\in H$, and any join subword $w'$ of $w$, we have $\ell(w')\leq N$. Let $K=(N+1)/2$ and we will prove that \[d_S(gu_1^{i}u_2^{j},H)\geq \frac{\abs{i}+\abs{j}}{K}-\abs{g}_S-1.\] 

Let $m=d_S(gu_1^{i}u_2^{j},H)$. Then there is an element $g_1$ in $G_\Gamma$ with $\abs{g_1}_S=m$ and $h$ in $H$ such that $h=gu_1^{i}u_2^{j}g_1$. Since $u_1^{i}u_2^{j}$ is an element in a join subgroup of $G_\Gamma$ and $\abs{g_1}_S=m$, then $h$ can be represented by a reduced word $w$ that is a product of at most $(\abs{g}_S+1+m)$ join subwords. Also, the length of each join subword of $w$ is bounded above by $N$. Therefore, the length of $w$ is bounded above by $N\bigl(\abs{g}_S+m+1\bigr)$. Also, \[\ell(w)\geq \abs{u_1^{i}u_2^{j}}_S-\abs{g_1}_S-\abs{g}_S\geq 2(\abs{i}+\abs{j})-m-\abs{g}_S.\] 

This implies that \[2(\abs{i}+\abs{j})-m-\abs{g}_S\leq N\bigl(\abs{g}_S+m+1\bigr).\]

Therefore, \[d_S(gs_1^{i}s_2^{j},H)=m\geq \frac{2(\abs{i}+\abs{j})}{N+1}-\abs{g}_S-\frac{N}{N+1}\geq \frac{\abs{i}+\abs{j}}{K}-\abs{g}_S-1.\]
\end{proof}

\begin{lem}
\label{ipt}
Let $\Gamma$ be a non-join $\mathcal{CFS}$ connected graph with the vertex set $S$. Let $C$ be a component of $\Gamma^4$ whose support is the entire vertex set of $\Gamma$. Let $H$ be a finitely generated join-free subgroup of $G_{\Gamma}$ and $h$ an arbitrary element in $H$. There is a number $L=L(H,C)\geq 1$ such that the following holds. Let $m\geq L^2$ an integer and $u=st$, where $(s,t)$ is a pair of non-adjacent vertices in some induced four-cycle $Q_0$ of $\Gamma$ that corresponds to a vertex in $C$. There is a path $\alpha$ outside the $(m/L-L)$--neighborhood of $H$ connecting $u^m$ and $hu^m$ with the length bounded above by $Lm$. 
\end{lem}

\begin{proof}
Let $M=\diam(C)$, $K$ the positive integer as in Lemma \ref{ltran} and $k=\abs{h}_S$. Let $L=2(k+1)(M+2)+K+k+M+1$. Choose a reduced word \[w=s_1s_2\cdots s_k\text{, where $s_i\in S$,}\] that represents element $h$. Since the support of the component $C$ of $\Gamma^4$ is the collection of vertices of $\Gamma$, for each $i\in \{1,2,\cdots, k\}$ there is a four-cycle $Q_i$ that corresponds to a vertex of the component $C$ of $\Gamma^4$ such that $Q_i$ contains the vertex $s_i$. Let $(a_i,b_i)$ be a pair of non-adjacent vertices in $Q_i$, $u_i=a_ib_i$ and $w_i=s_1s_2\cdots s_i$. Then the length of each word $w_i$ is bounded above by $k$, $w_{i+1}=w_is_i$, and $w_k=w$ that represents element $h$. 

We now construct a path $\alpha_0$ outside the $(m/L-L)$--neighborhood of $H$ connecting $u^m$ and $w_1u_1^m$ with the length bounded above by $2(M+2)m$. Since $M=\diam{C}$, we can choose positive integer $n\leq M$ and $n+1$ four-cycles $P_0, P_1, \cdots, P_n$ that corresponds to a vertex of the component $C$ of $\Gamma^4$ such that the following conditions hold:
\begin{enumerate}
\item $P_0=Q_0$ contains the pair of non-adjacent vertices $(s,t)$ and let $v_0=u$.
\item $P_n=Q_1$ contains the pair of non-adjacent vertices $(a_1,b_1)$ and let $v_{n+1}=u_1$.
\item $P_{j-1}$ and $P_{j}$ share an pair of non-adjacent vertices $(c_j,d_j)$, where $j\in \{1, 2, \cdots, n\}$ and let $v_j=c_jd_j$.
\end{enumerate}
For each $j\in \{0, 1, 2, \cdots, n\}$ let $\beta_j$ be a path connecting $v_j^m$ and $v_{j+1}^m$ of length $2m$ with vertices \[v_j^m, v_j^mv_{j+1}, v_j^mv_{j+1}^2, \cdots,v_j^mv_{j+1}^m, v_j^{m-1}v_{j+1}^m,v_j^{m-2}v_{j+1}^m, \cdots,v_{j+1}^m.\]
 
By Lemma \ref{ltran} the above vertices must lie outside the $(m/K-1)$--neighborhood of $H$. Therefore, these vertices also lies outside the $(m/L-L)$--neighborhood of $H$. Therefore, $\beta_j$ is a path outside the $(m/L-L)$--neighborhood of $H$ connecting $v_j^m$ and $v_{j+1}^m$. Since $w_1u_1^m=s_1u_1^m=u_1^ms_1$, then we can connect $u_1^m$ and $w_1u_1^m$ by an edge $\beta_{n+1}$ labelled by $s_1$. Let $\alpha_0=\beta_0\cup\beta_1\cup\cdots\cup \beta_n\cup\beta_{n+1}$. Then, it is obvious that the path $\alpha_0$ outside the $(m/L-L)$--neighborhood of $H$ connecting $u^m$ and $w_1u_1^m$ with the length bounded above by $2(M+2)m$. 

By similar constructions as above, for each $i\in \{1,2,\cdots, k-1\}$ there is a path $\alpha_i$ outside the $(m/L-L)$--neighborhood of $H$ connecting $w_iu_i^m$ and $w_{i+1}u_{i+1}^m$ with the length bounded above by $2(M+2)m$. We can also construct a path $\alpha_k$ outside the $(m/L-L)$--neighborhood of $H$ connecting $hu_k^m$ and $hu^m$ with the length bounded above by $2(M+1)m$. Let $\alpha=\alpha_0\cup\alpha_1\cup\cdots\cup \alpha_k$. Then, it is obvious that the path $\alpha$ outside the $(m/L-L)$--neighborhood of $H$ connecting $u^m$ and $hu^m$ with the length bounded above by $2(k+1)(M+2)m$. By the choice of $L$ we observe that the length of $\alpha$ is also bounded above by $Lm$.
\end{proof}

We now prove the quadratic subgroup divergence of join-free subgroups in $\mathcal{CFS}$ right-angled Coxeter groups.

\begin{thm}
Let $\Gamma$ be a non-join connected $\mathcal{CFS}$ graph with the vertex set $S$. Let $H$ be a finitely generated join-free subgroup of $G_{\Gamma}$. Then the subgroup divergence of $H$ in $G_{\Gamma}$ is exactly quadratic. 
\end{thm}

\begin{proof}
By the work of Cashen in \cite{Cashen2018} and Theorem D in \cite{JDT2018}, the subgroup divergence of $H$ in $G_{\Gamma}$ is at least quadratic. Therefore, we only need to prove that the subgroup divergence of $H$ in $G_{\Gamma}$ is at most quadratic. Since $\Gamma$ is a $\mathcal{CFS}$ graph, there is a component $C$ of $\Gamma^4$ whose support is the entire vertex set of $\Gamma$. Let $L=L(H,C)$ be a constant as in Lemma \ref{ipt} and $h$ an arbitrary infinite order group element in $H$. Since each cyclic subgroup in a $\CAT(0)$ group is undistorted, there is a positive integer $L_1$ such that \[\abs{h^k}_S\geq \frac{\abs{k}}{L_1}-L_1 \text{ for each integer $k$}.\] Let $\{\sigma^n_{\rho}\}$ be the subspace divergence of $H$ in the Cayley graph $\Sigma^{(1)}_{\Gamma}$. We will prove that function $\sigma^n_{\rho}(r)$ is bounded above by some quadratic function for each $n\geq 2$ and $\rho \in (0,1]$.

Choose a positive integer $m\in \bigl[L(L+r),2L(L+r)\bigr]$ and a group element $u=st$, where $(s,t)$ is a pair of non-adjacent vertices in a four-cycle $Q_0$ of $\Gamma$ that corresponds to a vertex in $C$. Then, there is a path $\alpha_0$ outside the $(m/L-L)$--neighborhood of $H$ connecting $u^m$ and $hu^m$ with the length bounded above by $Lm$. It is obvious that the path $\alpha_0$ also lies outside the $r$--neighborhood of $H$ by the choice of $m$. Choose a positive integer $k$ which lies between $L_1\bigl(nr+16L(L+r)+L_1\bigr)$ and $L_1\bigl(nr+16L(L+r)+L_1+1\bigr)$. Let $\alpha=\alpha_0\cup h\alpha_0\cup h^2\alpha_0\cup\cdots\cup h^{k-1}\alpha_0$. Then, $\alpha$ is a path outside the $r$--neighborhood of $H$ connecting $u^m$, $h^ku^m$ with the length bounded above by $kLm$. By the choice of $k$ and $m$, the length of $\alpha$ is bounded above by $2L_1L^2(L+r)\bigl(nr+16L(L+r)+L_1+1\bigr)$.

Since $r\leq d_S(u^m, H)\leq 2m$, then there is a path $\gamma_1$ outside $N_r(H)$ connecting $u^m$ and some point $x\in \partial N_r(H)$ such that the length of $\gamma_1$ is bounded above by $2m$. By the choice of $m$, the length of $\gamma_1$ is also bounded above by $4L(L+r)$. Similarly, there is a path $\gamma_2$ outside $N_r(H)$ connecting $h^ku^m$ and some point $y\in \partial N_r(H)$ such that the length of $\gamma_2$ is bounded above by $4L(L+r)$. Let $\bar{\alpha}=\gamma_1\cup\alpha\cup\gamma_2$ then $\bar{\alpha}$ is a path outside $N_r(H)$ connecting $x$, $y$ and the length of $\bar{\alpha}$ is bounded above by $2L_1L^2(L+r)\bigl(nr+16L(L+r)+L_1+1\bigr)+8L(L+r)$. Therefore, for each ~$\rho \in (0,1]$
\[d_{\rho r}(x,y)\leq 2L_1L^2(L+r)\bigl(nr+16L(L+r)+L_1+1\bigr)+8L(L+r).\]

Also, \begin{align*}d_S(x,y)&\geq d_S(u^m, h^ku^m)-d_S(u^m, x)-d_S(h^ku^m,y)\\&\geq \bigl(\abs{h^k}_S-4m\bigr)-4L(L+r)-4L(L+r)\geq \frac{k}{L_1}-L_1-16L(L+r)\\&\geq \bigl(nr+16L(L+r)\bigr)-16L(L+r)\geq nr.\end{align*}

 Thus, for each ~$\rho \in (0,1]$
\[\sigma_{\rho}^n(r)\leq 2L_1L^2(L+r)\bigl(nr+16L(L+r)+L_1+1\bigr)+8L(L+r).\] 

This implies that the subgroup divergence of $H$ in $A_{\Gamma}$ is at most quadratic. Therefore, the subgroup divergence of $H$ in $A_{\Gamma}$ is exactly quadratic.

\end{proof}

\subsection{Higher-degree polynomial subgroup divergence}

We first review the concept of the divergence of geodesic spaces and finitely generated groups in \cite{MR1254309}.
\begin{defn}
Let $X$ be a geodesic space and $x_0$ one point in $X$. For each $\rho \in (0,1]$, we define a function $\delta_{\rho}\!:[0, \infty)\to [0, \infty)$ as follows: 

For each $r$, let $\delta_{\rho}(r)=\sup d_{\rho r}(x_1,x_2)$ where the supremum is taken over all $x_1, x_2 \in S_r(x_0)$ such that $d_{\rho r}(x_1, x_2)<\infty$.

The family of functions $\{\delta_{\rho}\}$ is the \emph{divergence} of $X$ with respect to the point $x_0$, denoted $Div_{X,x_0}$.

In \cite{MR1254309}, Gersten show that the divergence $Div_{X,x_0}$ is, up to the relation ~$\sim$, a quasi-isometry invariant which is independent of the chosen basepoint. The \emph{divergence} of $X$, denoted $Div_X$, is then, up to the relation ~$\sim$, the divergence $Div_{X,x_0}$ for some point $x_0$ in $X$. 

If the space $X$ has the \emph{geodesic extension property} (i.e. any finite geodesic segment can be extended to an infinite geodesic ray), then it is not hard to show that $\delta_{\rho}\sim\delta_{1}$ for each $\rho \in (0,1]$. In this case, we can consider the divergence of $X$ as the function $\delta_1$. 

The \emph{divergence} of a finitely generated group $G$, denoted $Div(G)$, is the divergence of its Cayley graphs.
\end{defn}


The following definition was introduced by Levcovitz in \cite{IL} to study divergence in Coxeter groups.
\begin{defn}
Let $\Gamma$ be a finite, connected, simplicial graph. A pair of non-adjacent vertices $(s, t)$ is \emph{rank 1} if $s$ and $t$ are not contained in some induced four-cycle of $\Gamma$. Additionally, $(s, t)$ is \emph{rank $n$} if either every pair of non-adjacent vertices $(s_1, s_2)$ with $s_1, s_2 \in \Lk(s)$ is rank $n-1$ or every pair of non-adjacent vertices $(t_1, t_2)$ with $t_1, t_2 \in \Lk(t)$ is rank $n-1$.
\end{defn}

\begin{prop}
\label{pds1}
Let $\Gamma$ be a finite, connected, simplicial graph and $n$ a positive integer. There is a polynomial $f_n$ of degree $n$ such that the following hold. Let $(s,t)$ be a rank $n$ pair of vertices of $\Gamma$ and let $H_1$, $H_2$ be two hyperplanes of $\Sigma_\Gamma$ of types $s$, $t$ respectively such that their supports intersect. Let $p$ be a vertex in the intersection between two support of $H_1$ and $H_2$. The length of the any path from $H_1$ to $H_2$ which avoids the ball $B(p,r)$ is bounded below by $f_n(r)$. 
\end{prop}

The above proposition is a result from \cite{IL}. More precisely, two hyperplanes $H_1$ and $H_2$ are degree $n$ $M$-separated in the sense of Definition 6.1 in \cite{IL} (see the proof of Theorem 7.9 in \cite{IL}). Therefore, there is a polynomial $g_n$ of degree $n$ such that the length of any path from $H_1$ to $H_2$ which avoids the ball $B(p,r)$ is bounded below by $g_n(r)$ (see Theorem 6.2 in \cite{IL}). Since the number of rank $n$ pair of vertices in $\Gamma$ is finite, we can choose a universal polynomial $f_n$ of degree $n$ as in the above lemma.

\begin{thm}[Theorem 7.9 in \cite{IL}]
Let $\Gamma$ be a finite, connected, simplicial graph. Suppose $\Gamma$ contains a rank $n$ pair $(s, t)$, then $Div(W_Γ)$ is bounded below
by a polynomial of degree $n+1$.
\end{thm}

\begin{figure}
\begin{tikzpicture}[scale=0.65]

\draw (1,1) node[circle,fill,inner sep=1pt, color=black](1){} -- (2,2) node[circle,fill,inner sep=1pt, color=black](1){}-- (3,1) node[circle,fill,inner sep=1pt, color=black](1){}-- (2,0) node[circle,fill,inner sep=1pt, color=black](1){} -- (1,1) node[circle,fill,inner sep=1pt, color=black](1){}; 

\draw (2,2) node[circle,fill,inner sep=1pt, color=black](1){} -- (4,1) node[circle,fill,inner sep=1pt, color=black](1){}-- (2,0) node[circle,fill,inner sep=1pt, color=black](1){};

\draw (2,2) node[circle,fill,inner sep=1pt, color=black](1){} -- (5,1) node[circle,fill,inner sep=1pt, color=black](1){}-- (2,0) node[circle,fill,inner sep=1pt, color=black](1){};


\draw (2,2) node[circle,fill,inner sep=1pt, color=black](1){} -- (8,1) node[circle,fill,inner sep=1pt, color=black](1){}-- (2,0) node[circle,fill,inner sep=1pt, color=black](1){};

\draw (4,1) node[circle,fill,inner sep=1pt, color=black](1){} -- (3,-1) node[circle,fill,inner sep=1pt, color=black](1){}-- (2,-1) node[circle,fill,inner sep=1pt, color=black](1){};

\draw (5,1) node[circle,fill,inner sep=1pt, color=black](1){} -- (4,-1) node[circle,fill,inner sep=1pt, color=black](1){}-- (3,-1) node[circle,fill,inner sep=1pt, color=black](1){};

\draw (8,1) node[circle,fill,inner sep=1pt, color=black](1){} -- (7,-1) node[circle,fill,inner sep=1pt, color=black](1){}-- (6,-1) node[circle,fill,inner sep=1pt, color=black](1){};

\draw[densely dotted] (7,1) -- (6.5,0);

\draw (6.5,0) -- (6,-1);

\draw[densely dotted] (4,-1) node[circle,fill,inner sep=1pt, color=black](1){}-- (6,-1) node[circle,fill,inner sep=1pt, color=black](1){};

\draw[densely dotted] (5.5,1) -- (7,1) node[circle,fill,inner sep=1pt, color=black](1){};

\draw (1,1) node[circle,fill,inner sep=1pt, color=black](1){} -- (2,-1) node[circle,fill,inner sep=1pt, color=black](1){}-- (3,1) node[circle,fill,inner sep=1pt, color=black](1){};

\node at (2,-0.25) {$b_0$};

\node at (2,2.25) {$a_0$};

\node at (0.75,1) {$b_1$};

\node at (3.3,1) {$a_1$};

\node at (4.3,1) {$t_1$};

\node at (5.3,1) {$t_2$};

\node at (8.4,0.6) {$t_{d-2}$};


\node at (2,-1.4) {$b_2$};

\node at (3,-1.4) {$b_3$};

\node at (4,-1.4) {$b_4$};

\node at (6,-1.4) {$b_{d-1}$};

\node at (7,-1.4) {$b_d$};




\draw (2,2) node[circle,fill,inner sep=1pt, color=black](1){} -- (0,1) node[circle,fill,inner sep=1pt, color=black](1){}-- (2,0) node[circle,fill,inner sep=1pt, color=black](1){};

\draw (2,2) node[circle,fill,inner sep=1pt, color=black](1){} -- (-1,1) node[circle,fill,inner sep=1pt, color=black](1){}-- (2,0) node[circle,fill,inner sep=1pt, color=black](1){};


\draw (2,2) node[circle,fill,inner sep=1pt, color=black](1){} -- (-4,1) node[circle,fill,inner sep=1pt, color=black](1){}-- (2,0) node[circle,fill,inner sep=1pt, color=black](1){};

\draw (0,1) node[circle,fill,inner sep=1pt, color=black](1){} -- (1,-1) node[circle,fill,inner sep=1pt, color=black](1){}-- (2,-1) node[circle,fill,inner sep=1pt, color=black](1){};

\draw (-1,1) node[circle,fill,inner sep=1pt, color=black](1){} -- (0,-1) node[circle,fill,inner sep=1pt, color=black](1){}-- (1,-1) node[circle,fill,inner sep=1pt, color=black](1){};

\draw (-4,1) node[circle,fill,inner sep=1pt, color=black](1){} -- (-3,-1) node[circle,fill,inner sep=1pt, color=black](1){}-- (-2,-1) node[circle,fill,inner sep=1pt, color=black](1){};

\draw[densely dotted] (-3,1) -- (-2.5,0);

\draw (-2.5,0) -- (-2,-1);

\draw[densely dotted] (0,-1) node[circle,fill,inner sep=1pt, color=black](1){}-- (-2,-1) node[circle,fill,inner sep=1pt, color=black](1){};

\draw[densely dotted] (-1.5,1) -- (-3,1) node[circle,fill,inner sep=1pt, color=black](1){};

\draw (1,1) node[circle,fill,inner sep=1pt, color=black](1){} -- (2,-1) node[circle,fill,inner sep=1pt, color=black](1){}-- (3,1) node[circle,fill,inner sep=1pt, color=black](1){};

\node at (-0.3,1) {$s_1$};

\node at (-1.3,1) {$s_2$};

\node at (-4.6,0.6) {$s_{d-2}$};



\node at (1,-1.4) {$a_3$};

\node at (0,-1.4) {$a_4$};

\node at (-2,-1.4) {$a_{d-1}$};

\node at (-3,-1.4) {$a_d$};

\draw (1,1) node[circle,fill,inner sep=1pt, color=black](1){} -- (2,3) node[circle,fill,inner sep=1pt, color=black](1){}-- (3,1) node[circle,fill,inner sep=1pt, color=black](1){}; 

\node at (2,3.25) {$c$};

\node at (2,-2.25) {$\Omega_d$};


\draw (10,1) node[circle,fill,inner sep=1pt, color=black](1){} -- (11,2) node[circle,fill,inner sep=1pt, color=black](1){}-- (12,1) node[circle,fill,inner sep=1pt, color=black](1){}-- (11,0) node[circle,fill,inner sep=1pt, color=black](1){} -- (10,1) node[circle,fill,inner sep=1pt, color=black](1){}; 

\draw (11,2) node[circle,fill,inner sep=1pt, color=black](1){} -- (13,1) node[circle,fill,inner sep=1pt, color=black](1){}-- (11,0) node[circle,fill,inner sep=1pt, color=black](1){};

\draw (11,2) node[circle,fill,inner sep=1pt, color=black](1){} -- (14,1) node[circle,fill,inner sep=1pt, color=black](1){}-- (11,0) node[circle,fill,inner sep=1pt, color=black](1){};

\draw (11,2) node[circle,fill,inner sep=1pt, color=black](1){} -- (18,1) node[circle,fill,inner sep=1pt, color=black](1){}-- (11,0) node[circle,fill,inner sep=1pt, color=black](1){};

\draw (11,2) node[circle,fill,inner sep=1pt, color=black](1){} -- (17,1) node[circle,fill,inner sep=1pt, color=black](1){}-- (11,0) node[circle,fill,inner sep=1pt, color=black](1){};

\draw (13,1) node[circle,fill,inner sep=1pt, color=black](1){} -- (12,-1) node[circle,fill,inner sep=1pt, color=black](1){}-- (11,-1) node[circle,fill,inner sep=1pt, color=black](1){};

\draw (14,1) node[circle,fill,inner sep=1pt, color=black](1){} -- (13,-1) node[circle,fill,inner sep=1pt, color=black](1){}-- (12,-1) node[circle,fill,inner sep=1pt, color=black](1){};

\draw (17,1) node[circle,fill,inner sep=1pt, color=black](1){} -- (16,-1) node[circle,fill,inner sep=1pt, color=black](1){}-- (15,-1) node[circle,fill,inner sep=1pt, color=black](1){};

\draw[densely dotted] (16,1) -- (15.5,0);

\draw (15.5,0) -- (15,-1);

\draw[densely dotted] (13,-1) node[circle,fill,inner sep=1pt, color=black](1){}-- (15,-1) node[circle,fill,inner sep=1pt, color=black](1){};

\draw[densely dotted] (14.5,1) -- (16,1) node[circle,fill,inner sep=1pt, color=black](1){};

\draw (10,1) node[circle,fill,inner sep=1pt, color=black](1){} -- (11,-1) node[circle,fill,inner sep=1pt, color=black](1){}-- (12,1) node[circle,fill,inner sep=1pt, color=black](1){};

\node at (11,-0.25) {$b_0$};

\node at (11,2.25) {$a_0$};

\node at (9.75,1) {$b_1$};

\node at (12.3,1) {$a_1$};

\node at (13.3,1) {$t_1$};

\node at (14.3,1) {$t_2$};

\node at (17.6,0.5) {$t_{m-3}$};

\node at (18.8,1) {$t_{m-2}$};

\node at (11,-1.4) {$b_2$};

\node at (12,-1.4) {$b_3$};

\node at (13,-1.4) {$b_4$};

\node at (15,-1.4) {$b_{m-2}$};

\node at (16.5,-1.4) {$b_{m-1}$};




\node at (14,-2.25) {$\Gamma_{m-1}$};

\end{tikzpicture}

\caption{}
\label{asecond}
\end{figure}

We now construct right-angled Coxeter groups with join-free subgroups of different subgroup divergence. More precisely, for each $d \geq 3$ let $\Omega_d$ be a graph in Figure \ref{asecond}. We will construct non-virtually cyclic join-free subgroups with subgroup divergence of polynomial of degree $m$ for $2\leq m \leq d$. We remark that the graphs $\Gamma_m$ in Figure \ref{asecond} were introduced by Dani-Thomas \cite{MR3314816} to study divergence of right-angled Coxeter groups and each graph $\Omega_d$ in Figure \ref{asecond} is a variation of the graph $\Gamma_d$. We now prepare some lemmas and propositions that help with the construction of the desired join-free subgroups in $G_{\Omega_d}$.

\begin{lem}
\label{lds2}
For each $d\geq 3$ let $\Omega_d$ be a graph as in Figure \ref{asecond}. For each $3\leq m\leq d$ all pairs $(a_m,b_m)$, $(a_m,c)$, $(b_m,c)$ are rank $m-1$. 
\end{lem}

\begin{proof}
We first prove that for $3\leq\ell\leq k\leq d$ each pair of non-adjacent vertices in $\Lk(a_k)$ is rank $\ell-2$. We will prove this by induction on $\ell$. For $\ell=3$ we observe that each pair of non-adjacent vertices in $\Lk(a_k)$ ($k\geq 3$) are not contained in some induced four-cycle of $\Gamma$. Therefore, these pairs are all rank 1. 

Assume that there is $4\leq\ell_0\leq d-1$ such that each pair of non-adjacent vertices in $\Lk(a_k)$ ($\ell_0\leq k\leq d$) is rank $\ell_0-2$. We need to prove that each pair of non-adjacent vertices in $\Lk(a_k)$ ($\ell_0+1\leq k\leq d$) is rank $\ell_0-1$. We observe that $\Lk(a_k)=\{s_{k-2}, a_{k-1}, a_{k+1}\}$ ($\ell_0+1\leq k\leq d-1$). By hypothesis induction, each pair of non-adjacent vertices in $\Lk(a_{k-1})$, $\Lk(a_{k+1})$ is rank $\ell_0-2$. Therefore, all pairs of non-adjacent vertices $(a_{k-1},s_{k-2})$, $(a_{k+1},s_{k-2})$, $(a_{k-1},a_{k+1})$ are rank $\ell_0-1$. In other word, each pair of non-adjacent vertices in $\Lk(a_k)$ ($\ell_0+1\leq k\leq d-1$) is rank $\ell_0-1$. For the case $k=d$ we see that $\Lk(a_k)=\Lk(a_d)=\{a_{d-1},s_{d-2}\}$ and each pair of non-adjacent vertices in $\Lk(a_{d-1})$ is rank $\ell_0-2$ by hypothesis induction. Therefore, pair of non-adjacent vertices in $\Lk(a_{d})$ is rank $\ell_0-1$ by hypothesis induction. Thus, for $3\leq\ell\leq k\leq d$ each pair of non-adjacent vertices in $\Lk(a_k)$ is rank $\ell-2$. In particular, each pair of non-adjacent vertices in $\Lk(a_m)$ ($3\leq m\leq d$) is rank $m-2$. By a similar argument, each pair of non-adjacent vertices in $\Lk(b_m)$ ($3\leq m\leq d$) is rank $m-2$. This implies that for each $3\leq m\leq d$ all pairs $(a_m,b_m)$, $(a_m,c)$, $(b_m,c)$ are rank $m-1$. 
\end{proof}

The following proposition is a direct result of Proposition \ref{pds1} and Lemma \ref{lds2}.

\begin{prop}
\label{pl}
For each $d\geq 3$ and $3\leq m \leq d$ let $\Omega_d$ be the graph as in Figure \ref{asecond} and $H_d^m$ the subgroup of $G_{\Omega_d}$ generated by $c$, $a_m$, and $b_m$. Then the subgroup divergence of $H_d^m$ in $G_{\Omega_d}$ is bounded below by a polynomial of degree $m$.
\end{prop}

\begin{proof}
Let $\{\sigma^n_{\rho}\}$ be the subspace divergence of $H_d^m$ in the Caley graph $\Sigma^{(1)}_{\Omega_d}$. Let $f_{m-1}$ be the polynomials of degree $m-1$ as in Proposition \ref{pds1}. We will prove that for each $n\geq 8$ and $\rho \in (0,1]$
\[\sigma^n_{\rho}(r)\geq (r-1)f_{m-1}(\rho r) \text{ for each $r>1$.}\]

Let $u$ and $v$ be an arbitrary pair of points in $\partial N_r(H)$ such that $d_r(u, v)<\infty$ and $d_S(u,v)\geq nr$. Let $\gamma$ be an arbitrary path that lies outside the $\rho r$--neighborhood of $H$ connecting $u$ and $v$. We will prove that the length of $\gamma$ is bounded below by $(r-1)f_{m-1}(\rho r)$.

Let $\gamma_1$ be a geodesic of length $r$ in $\Sigma^{(1)}_{\Omega_d}$ connecting $u$ and some point $x$ in $H_d^m$. Let $\gamma_2$ be another geodesic of length $r$ in $\Sigma^{(1)}_{\Omega_d}$ connecting $v$ and some point $y$ in $H_d^m$. Let $\alpha$ be a geodesic in $\Sigma^{(1)}_{\Omega_d}$ connecting $x$ and $y$. Obviously, each edge of $\alpha$ is labelled by $a_m$, $b_m$, or $c$. This implies that two hyperplanes determined by two different edges in $\alpha$ do not intersect. Since $d_S(x,y)\geq d_S(u,v)-2r\geq (n-2)r\geq 6r$, there is a subpath $\beta$ with length bounded below by $r$ of $\alpha$ such that $\beta\cap\bigl(B(x,2r)\cup B(y,2r)\bigr)=\emptyset$. Also, the lengths of $\gamma_1$ and $\gamma_2$ are both $r$. This implies that each hyperplane determined by edge in $\beta$ does not intersect $\gamma_1\cup\gamma_2$. Therefore, each hyperplane determined by edge in $\beta$ must intersect $\gamma$.

Assume that the path $\beta$ is the concatenation of edges $e_1, e_2,\cdots, e_k$, $k\geq r$ and let $H_i$ be the hyperplane determined by edge $e_i$. Therefore, for each $i \in \{1,2,\cdots, k-1\}$ the $(i+1)^{th}$ vertex $p_i$ of $\beta$ lies in the intersection of the support of $H_i$ and $H_{i+1}$. For each $i\in \{1,2,\cdots, k\}$ let $x_i$ be a point in $H_i\cap \gamma$. Let $\gamma_i$ be the subpath of $\gamma$ connecting $x_i$ and $x_{i+1}$ for each $i \in \{1,2,\cdots, k-1\}$. Therefore, each $\gamma_i$ is a path from $H_i$ to $H_{i+1}$ which avoids the ball $B(p_i,\rho r)$. Therefore, the length of each $\gamma_i$ is bounded below by $f_{m-1}(\rho r)$ by Proposition \ref{pds1} and Lemma \ref{lds2}. This implies that the length of $\gamma$ is bounded below by $(k-1)f_{m-1}(\rho r)$. Also, $k\geq r$. Therefore, the length of $\gamma$ is bounded below by $(r-1)f_{m-1}(\rho r)$. Thus, $\sigma^n_{\rho}(r)\geq (r-1)f_{m-1}(\rho r)$. Therefore, the subgroup divergence of $H_d^m$ in $G_{\Omega_d}$ is bounded below by a polynomial of degree $m$. 
\end{proof}

The following lemma contributes to the proof of the upper bound of our subgroup divergences.

\begin{lem}
\label{nl}
For each $d\geq 3$ and $3\leq m\leq d$ there is a polynomial $g_{m-1}$ of degree $m-1$ such that the following holds. Let $\alpha$ be a geodesic ray based at $e$ that is labelled by $a_1b_1a_1b_1\cdots$. Let $\beta$ be a geodesic ray based at $e$ that is labelled by $b_{m-1}t_{m-2}b_{m-1}t_{m-2}\cdots$. Then for each $r>0$ there is a path outside $N_r(H_d^m)$ connecting $\alpha(r)$ and $\beta(r)$ with length bounded above by $g_{m-1}(r)$. 
\end{lem}

\begin{proof}
Let $\Gamma_{m-1}$ be a subgraph of $\Omega_d$ as in Figure \ref{asecond}. Let $S$, $S'$ be vertex sets of $\Omega_d$, $\Gamma_{m-1}$ respectively. Obviously, $\alpha$ and $\beta$ be two geodesic rays in the 1-skeleton $\Sigma^{(1)}_{\Gamma_{m-1}}$ of the Davis complex $\Sigma_{\Gamma_{m-1}}$. Since the Cayley graph $\Sigma^{(1)}_{\Gamma_{m-1}}$ has geodesic extension property and the divergence of $\Sigma^{(1)}_{\Gamma_{m-1}}$ is a polynomial of degree $m-1$ (see Section 5 in \cite{MR3314816}), there is a polynomial $g_{m-1}$ of degree $m-1$ such that for each $r>0$ there is a path $\gamma_r$ in $\Sigma^{(1)}_{\Gamma_{m-1}}$ with length bounded above by $g_{m-1}(r) $ connecting $\alpha(r)$, $\beta(r)$ and $\gamma_r$ avoids the ball $B(e,r)$ in $\Sigma^{(1)}_{\Gamma_{m-1}}$. We now prove that each $\gamma_r$ also lies outside $N_r(H_d^m)$ by showing that its vertices lie outside $N_r(H_d^m)$.

Let $\Phi\!:G_{\Omega_d}\to G_{\Gamma_{m-1}}$ be the group homomorphism induced by mapping each vertex of $\Gamma_{m-1}$ to itself and each vertex outside $\Gamma_{m-1}$ to $e$. It is not hard to check the following:
\begin{enumerate}
\item The map $\Phi$ is a well-defined group homomorphism.
\item $\Phi(u)=u$ for each $u$ in $G_{\Gamma_{m-1}}$ and $\Phi(h)=e$ for each $h$ in $H_d^m$.
\item $\abs{\Phi(g)}_{S'}\leq\abs{g}_S$ for each $g$ in $G_{\Omega_d}$.
\end{enumerate}

For each vertex $u$ in $\gamma_r$, $u$ is a group element in $G_{\Gamma_{m-1}}$ with $\abs{u}_{S'}\geq r$. Assume that $m=d_S(u,H_d^m)$. Then there is $h$ in $H_d^m$ such that $m=d_S(h,u)=\abs{h^{-1}u}_S$. Therefore, \[m=\abs{h^{-1}u}_S\geq\abs{\Phi(h^{-1}u)}_{S'}=\abs{u}_{S'}\geq r.\] This implies that each vertex in $\gamma_r$ lies outside $N_r(H_d^m)$. Therefore, each path $\gamma_r$ also lie outside $N_r(H_d^m)$.

\end{proof}

\begin{prop}
\label{pu}
For each $d\geq 3$ and $3\leq m \leq d$ let $\Omega_d$ be the graph as in Figure \ref{asecond} and $H_d^m$ the subgroup of $G_{\Omega_d}$ generated by $c$, $a_m$, and $b_m$. Then the subgroup divergence of $H_d^m$ in $G_{\Omega_d}$ is bounded above by a polynomial of degree $m$.
\end{prop}

\begin{proof}
Let $\alpha$, $\beta$ be geodesic rays as in Lemma \ref{nl} and $g_{m-1}$ a polynomial of degree $m-1$ as in this lemma. Let $\{\sigma^n_{\rho}\}$ be the subspace divergence of $H_d^m$ in the Cayley graph $\Sigma^{(1)}_{\Omega_d}$. We will prove that function $\sigma^n_{\rho}(r)$ is bounded above by some quadratic function for each $n\geq 2$ and $\rho \in (0,1]$.

For each $r>1$ there is a path $\gamma_r$ outside $N_r(H_d^m)$ connecting $\alpha(r)$ and $\beta(r)$ with length bounded above by $g_{m-1}(r)$. Since the generator $b_m$ commutes with all edge labels of $\beta$, two points $\beta(r)$ and $b_m\beta(r)$ lie on the boundary of a 2-cell in $\Sigma_{\Omega_d}$. Therefore, there is a path $\alpha_1$ outside $N_r(H_d^m)$ connecting $\beta(r)$ and $b_m\beta(r)$ with length bounded above by 3. Similarly, the generator $c$ commutes with all edge labels of $\alpha$, two points $b_m\alpha(r)$ and $(b_mc)\alpha(r)$ lie on the boundary of a 2-cell in $\Sigma_{\Omega_d}$. Therefore, there is a path $\alpha_2$ outside $N_r(H_d^m)$ connecting $b_m\alpha(r)$ and $b_mc\alpha(r)$ with length bounded above by 3. Also $b_m\gamma_r$ is a path outside $N_r(H_d^m)$ connecting $b_m\alpha(r)$ and $b_m\beta(r)$ with length bounded above by $g_{m-1}(r)$. Therefore, $\eta_1=\gamma_r\cup\alpha_1\cup b_m\gamma_r\cup\alpha_2$ is a path outside $N_r(H_d^m)$ connecting $\alpha(r)$ and $(b_mc)\alpha(r)$ with length bounded above by $2g_{m-1}(r)+6$.

For each $n\geq 2$, let $k$ be an integer between $nr$ and $2nr$. Let \[\eta=\eta_1\cup(b_mc)\eta_1\cup(b_mc)^2\eta_1 \cup \cdots\cup(b_mc)^{k-1}\eta_1.\] Then, $\eta$ is a path outside $N_r(H_d^m)$ connecting $\alpha(r)$ and $(b_mc)^k\alpha(r)$ with length bounded above by $k\bigl(2g_{m-1}(r)+6\bigr)$. Therefore, \[d_{\rho r}\bigl(\alpha(r), (b_mc)^k\alpha(r)\bigr)\leq k\bigl(2g_{m-1}(r)+6\bigr)\leq 2nr\bigl(2g_{m-1}(r)+6\bigr)\]

Also, \[d_S\bigl(\alpha(r), (b_mc)^k\alpha(r)\bigr)\geq d_S\bigl(e,(b_mc)^k\bigr)-2r\geq 2k-2r\geq (2n-2)r\geq nr.\]
Therefore, $\sigma^n_{\rho}(r)\leq 2nr\bigl(2g_{m-1}(r)+6\bigr)$. This implies that the subgroup divergence of $H$ in $G_{\Omega_d}$ is bounded above by a polynomial of degree $m$.
\end{proof}

By using similar techniques as in Lemma \ref{nl} and Proposition \ref{pu}, we also obtain the following proposition.

\begin{prop}
\label{puq}
For each $d\geq 3$ let $\Omega_d$ be the graph as in Figure \ref{asecond} and $H_d^2$ the subgroup of $G_{\Omega_d}$ generated by $c$, $s_1$, and $t_1$. Then the subgroup divergence of $H_d^2$ in $G_{\Omega_d}$ is exactly a quadratic function.
\end{prop}

We are now ready for the main theorem in this section.

\begin{thm}
For each $d\geq 3$ let $\Omega_d$ be the graph as in Figure \ref{asecond}. Let $H_d^2$ be the subgroup of $G_{\Omega_d}$ generated by the set $\{c,s_1,t_1\}$. For each $3\leq m \leq d$ let $H_d^m$ the subgroup of $G_{\Omega_d}$ generated by the set $\{c,a_m,b_m\}$. Then for each $2\leq m\leq d$ subgroup $H_d^m$ is a join-free subgroup of $G_{\Omega_d}$, $H_d^m$ is isomorphic to the group $F=\presentation{s,t,u}{s^2=t^2=u^2=e}$, and the subgroup divergence of $H_d^m$ in $G_{\Omega_d}$ is a polynomial of degree $m$.
\end{thm}

\begin{proof}
We first consider the case $3\leq m \leq d$. It is not hard to see that each infinite order element $h$ in $H_d^m$ can be written as a reduced word $s_1s_2\cdots s_m$, where each $s_i$ belongs to the set $\{a_m,b_m,c\}$ and $s_i$, $s_{i+1}$ are two different elements in $\{a_m,b_m,c\}$. Therefore, $H_d^m$ is a join-busting subgroup. This implies that $H_d^m$ is a join-free subgroup. By Propositions \ref{pl} and \ref{pu}, the subgroup divergence of $H_d^m$ in $G_{\Omega_d}$ is a polynomial of degree $m$. By a similar argument the subgroup $H_d^2$ is also join-busting. Therefore, $H_d^2$ is also a join-free subgroup. The fact that the subgroup divergence of $H_d^2$ in $G_{\Omega_d}$ is a quadratic function can be seen from Proposition \ref{puq}. It is also obvious that all special subgroups are isomorphic to the group $F=\presentation{s,t,u}{s^2=t^2=u^2=e}$.
\end{proof}
 
\section*{Appendix A. Finite height subgroups}

We note that the proof of the strong quasiconvexity of finitely generated finite height subgroups in one-ended right-angled Artin groups was already given implicitly by the author in \cite{Tran2017}. We now generalize a part of that work in \cite{Tran2017} to provide necessary conditions for finite height subgroups of groups satisfying certain conditions (see Proposition~\ref{fhimpliessq} and Lemma~\ref{simplecase}). After that we give an explicit proof of the fact finitely generated finite height subgroups in one-ended right-angled Artin groups are always strongly quasiconvex. Finally we prove that finite height parabolic subgroups in right-angled Coxeter groups are also strongly quasiconvex. 

In the following proposition, we provide a necessary condition for infinite index finite height subgroups of groups satisfying certain conditions.

\begin{propappend}
\label{fhimpliessq}
Let $G$ be a group and suppose there is a collection $\mathcal{A}$ of subgroups of $G$ that satisfies the following conditions:
\begin{enumerate}
\item For each $A$ in $\mathcal{A}$ and $g$ in $G$ the conjugate $g^{-1}Ag$ also belongs to $\mathcal{A}$ and there is a finite sequence $$A=A_0, A_1, \cdots, A_n=g^{-1}Ag$$ of subgroups in $\mathcal{A}$ such that $A_{j-1}\cap A_j$ is infinite for each $j$;
\item For each $A$ in $\mathcal{A}$ each finite height subgroup of $A$ must be finite or have finite index in $A$.
\end{enumerate}
Then for each infinite index finite height subgroup $H$ of $G$ the intersection $H\cap A$ must be finite for all $A$ in $\mathcal{A}$.
\end{propappend}

\begin{proof}
 We assume for the contradiction that $H\cap A_0$ is infinite for some $A_0\in\mathcal{A}$. We claim that $H\cap g^{-1}A_{0}g$ has finite index in $g^{-1}A_{0}g$ for all $g\in G$. Since $H$ has finite height in $G$, the subgroup $H\cap A_0$ has finite height in $A_0$. Therefore, $H\cap A_0$ has finite index in $A_0$ by the hypothesis and our assumption. 

By the hypothesis, there is a finite sequence $A_{0}=A_0,A_1,\cdots, A_m=g^{-1}A_0g$ of subgroups in $\mathcal{A}$ such that $A_{i-1}\cap A_i$ is infinite for each $i\in \{1,2,\cdots,m\}$. Since $H\cap A_0$ has finite index in $A_0$ by the above argument and $A_0\cap A_1$ is infinite, $H\cap A_1$ is also infinite. By a similar argument as above, $H\cap A_1$ has finite index in $A_1$. Repeating this process, we have $H\cap g^{-1}A_0g$ has finite index in $g^{-1}A_0g$. In other words, $gHg^{-1}\cap A_{0}$ has finite index in $A_0$ for all $g\in G$.

Since $H$ has finite height in $G$, there is a number $n$ such that the intersection of any $(n+1)$ essentially distinct conjugates of $H$ is finite. Also $H$ has infinite index in $G$. Therefore, there is $n+1$ distinct elements $g_1, g_2,\cdots g_{n+1}$ such that $g_iH\neq g_j H$ for each $i\neq j$. Also, $g_i H g_i^{-1} \cap A_0$ has finite index in $A_0$ for each $i$. Then $(\cap g_i H g_i^{-1}) \bigcap A_0$ also has finite index in $A_0$. In particular, $\bigcap g_i H g_i^{-1}$ is infinite which is a contradiction. Therefore, the intersection $H\cap A$ must be finite for all $A$ in $\mathcal{A}$ and all $g$ in $G$.
\end{proof}

In the following lemma, we study finite height subgroups in certain direct product of groups.

\begin{lemappend}
\label{simplecase}
Let $G_1$ and $G_2$ be two groups such that each of them contains an infinite order elements. Let $H$ be a finite height subgroup of $G=G_1\times G_2$ and $H$ contains an infinite order element. Then $H$ must have finite index in $G$. In particular, if $H$ is an almost malnormal subgroup and $H$ contains an infinite order element, then $H=G$. 
\end{lemappend}

\begin{proof}
It suffices to show that $H\cap G_1$ has finite index in $G_1$ and $H\cap G_2$ has finite index in $G_2$. Let $h$ be an infinite order element of $H$. Then $h=g_1g_2$ where $g_1$ is a group element of $G_1$, $g_2$ is a group element of $G_2$, and either of them (say $g_1$) has infinite order. We claim that $g_1^{n_0}$ is an element of $H$ for some $n_0>0$. Otherwise, $(g_1^nH)_{n\geq 0}$ is an infinite sequence of distinct left cosets and $\bigcap g_1^nHg_1^{-n}$ contains the infinite order element $h$ which contradicts to the fact that $H$ has finite height.

We now claim that $H\cap G_2$ has finite index in $G_2$. Otherwise there is an infinite sequence $(k_n)_{n\geq 1}$ of groups elements in $G_2$ such that $(k_nH)_{n\geq 0}$ is an infinite sequence of distinct left cosets. Also subgroup $k_nHk_n^{-1}$ contains the infinite order element $g_1^{n_0}$ for all $n\geq 1$. This contradicts to the fact that $H$ has finite height. The subgroup $H\cap G_2$ has finite index in $G_2$. Since $G_2$ contains an infinite order element, $H$ contains an infinite order element of $G_2$. By a similar argument, $H\cap G_1$ has finite index in $G_1$. Therefore, $H$ must have finite index in $G$.
\end{proof}

In the following two propositions, we study finite height subgroups in right-angled Artin groups and right-angled Coxeter groups.

\begin{propappend}
Let $A_\Gamma$ be a one-ended right-angled Artin group and $H$ a finitely generated subgroup of $A_\Gamma$. Then $H$ is strongly quasiconvex if and only if $H$ has finite height.
\end{propappend}

\begin{proof}
Since all strongly quasiconvex subgroups have finite height (see Theorem 1.2 in \cite{Tran2017}), we only need to prove the opposite direction. We assume now that $H$ has finite height and $H$ is not the trivial subgroup. Let $\mathcal{A}$ be the collection of all parabolic subgroups of $A_\Gamma$ induced by some join subgraph. It is clear from the construction that $\mathcal{A}$ satisfies the first part of Condition (1) of Proposition~\ref{fhimpliessq}. The proof of the second part of this condition is almost identical to the proof of Lemma 8.17 in \cite{Tran2017}. By Lemma~\ref{simplecase} the collection $\mathcal{A}$ also satisfies the Condition (2) of Proposition~\ref{fhimpliessq}. If $H$ has finite index in $A_\Gamma$, $H$ is strongly quasiconvex trivially. Otherwise, $H\cap A$ must be trivial for all $A$ in $\mathcal{A}$. Therefore, $H$ is strongly quasiconvex by Corollary 1.17 in \cite{Tran2017} or Theorem B.1 in \cite{Genevois2017}. 
\end{proof}



\begin{propappend}
\label{fhpsracgs}
Let $G_\Gamma$ be a right-angled Coxeter group and $H$ a parabolic subgroup of $G_\Gamma$. Then $H$ is strongly quasiconvex if and only if $H$ has finite height.
\end{propappend}

\begin{proof}
We note that if a subgroup is strongly quasiconvex, then all its conjugates are also strongly quasiconvex. Similarly, if a subgroup has finite height, then all its conjugates also have finite height. Therefore, we can assume that $H$ is a special subgroup. Let $\Lambda$ be the induced subgraph of $\Gamma$ that defines $H$. Since all strongly quasiconvex subgroups have finite height (see Theorem 1.2 in \cite{Tran2017}), we only need to prove the opposite direction. We assume now that $H$ has finite height. Assume for a contradiction that $H$ is not strongly quasiconvex. By Proposition 4.9 in \cite{Genevois2017} or Theorem 7.5 in \cite{JDT2018} there is an induced $4$--cycle $\sigma$ such that $\Lambda\cap \sigma$ contains two non adjacent vertices, call $a_1$ and $a_2$ and $\sigma-(\sigma\cap \Lambda)$ contains at least a vertex, call $b_1$. Let $b_2$ be a opposite vertex of $b_1$ in $\sigma$ and let $g=b_1b_2$. Then it is clear that $g^mH\neq g^nH$ for $m\neq n$. Also, $\bigcap g^nHg^{-n}$ contains the infinite order element $a_1a_2$. This contradicts to the assumption that $H$ has finite height. Therefore, $H$ is a strongly quasiconvex subgroup.
\end{proof}

\bibliographystyle{alpha}
\bibliography{Tran}

\def\cprime{$'$}
\begin{thebibliography}{GMRS98}

\bibitem[AM15]{MR3365774}
Yago Antol\'{i}n and Ashot Minasyan.
\newblock Tits alternatives for graph products.
\newblock {\em J. Reine Angew. Math.}, 704:55--83, 2015.

\bibitem[BC12]{MR2874959}
Jason Behrstock and Ruth Charney.
\newblock Divergence and quasimorphisms of right-angled {A}rtin groups.
\newblock {\em Math. Ann.}, 352(2):339--356, 2012.

\bibitem[Bes]{Bestvina}
Mladen Bestvina.
\newblock Questions in geometric group theory.
\newblock \emph{M. Bestvina's home page}, 2004.

\bibitem[Cap09]{MR2665193}
Pierre-Emmanuel Caprace.
\newblock Buildings with isolated subspaces and relatively hyperbolic {C}oxeter
  groups.
\newblock {\em Innov. Incidence Geom.}, 10:15--31, 2009.

\bibitem[Cap15]{MR3450952}
Pierre-Emmanuel Caprace.
\newblock Erratum to ``{B}uildings with isolated subspaces and relatively
  hyperbolic {C}oxeter groups'' [ {MR}2665193].
\newblock {\em Innov. Incidence Geom.}, 14:77--79, 2015.

\bibitem[Cas]{Cashen2018}
Christopher~H. Cashen.
\newblock Morse subsets of {CAT(0)} spaces are strongly contracting.
\newblock Submitted. arXiv:1810.02119.

\bibitem[Dav08]{MR2360474}
Michael~W. Davis.
\newblock {\em The geometry and topology of {C}oxeter groups}, volume~32 of
  {\em London Mathematical Society Monographs Series}.
\newblock Princeton University Press, Princeton, NJ, 2008.

\bibitem[DGO17]{DGO}
Fran\c{c}ois Dahmani, Vincent Guirardel, and Denis~V. Osin.
\newblock Hyperbolically embedded subgroups and rotating families in groups
  acting on hyperbolic spaces.
\newblock {\em Mem. Amer. Math. Soc.}, 245(1156):v+152, 2017.

\bibitem[DL98]{MR1709956}
Warren Dicks and Ian~J. Leary.
\newblock On subgroups of {C}oxeter groups.
\newblock In {\em Geometry and cohomology in group theory ({D}urham, 1994)},
  volume 252 of {\em London Math. Soc. Lecture Note Ser.}, pages 124--160.
  Cambridge Univ. Press, Cambridge, 1998.

\bibitem[DT15a]{MR3314816}
Pallavi Dani and Anne Thomas.
\newblock Divergence in right-angled {C}oxeter groups.
\newblock {\em Trans. Amer. Math. Soc.}, 367(5):3549--3577, 2015.

\bibitem[DT15b]{MR3426695}
Matthew~Gentry Durham and Samuel~J. Taylor.
\newblock Convex cocompactness and stability in mapping class groups.
\newblock {\em Algebr. Geom. Topol.}, 15(5):2839--2859, 2015.

\bibitem[Gen]{Genevois2017}
Anthony Genevois.
\newblock Hyperbolicities in {CAT(0)} cube complexes.
\newblock Submitted. arXiv:1709.08843.

\bibitem[Ger94]{MR1254309}
S.~M. Gersten.
\newblock Quadratic divergence of geodesics in {${\rm CAT}(0)$} spaces.
\newblock {\em Geom. Funct. Anal.}, 4(1):37--51, 1994.

\bibitem[GMRS98]{MR1389776}
Rita Gitik, Mahan Mitra, Eliyahu Rips, and Michah Sageev.
\newblock Widths of subgroups.
\newblock {\em Trans. Amer. Math. Soc.}, 350(1):321--329, 1998.

\bibitem[HM95]{MR1314099}
Susan Hermiller and John Meier.
\newblock Algorithms and geometry for graph products of groups.
\newblock {\em J. Algebra}, 171(1):230--257, 1995.

\bibitem[HW10]{MR2646113}
Fr\'ed\'eric Haglund and Daniel~T. Wise.
\newblock Coxeter groups are virtually special.
\newblock {\em Adv. Math.}, 224(5):1890--1903, 2010.

\bibitem[Kim08]{MR2443098}
Sang-hyun Kim.
\newblock Co-contractions of graphs and right-angled {A}rtin groups.
\newblock {\em Algebr. Geom. Topol.}, 8(2):849--868, 2008.

\bibitem[KK14]{MR3192368}
Sang-Hyun Kim and Thomas Koberda.
\newblock The geometry of the curve graph of a right-angled {A}rtin group.
\newblock {\em Internat. J. Algebra Comput.}, 24(2):121--169, 2014.

\bibitem[KMT17]{KMT}
Thomas Koberda, Johanna Mangahas, and Samuel~J. Taylor.
\newblock The geometry of purely loxodromic subgroups of right-angled {A}rtin
  groups.
\newblock {\em Trans. Amer. Math. Soc.}, 369(11):8179--8208, 2017.

\bibitem[Lev18]{IL}
Ivan Levcovitz.
\newblock Divergence of {$\rm CAT(0)$} cube complexes and {C}oxeter groups.
\newblock {\em Algebr. Geom. Topol.}, 18(3):1633--1673, 2018.

\bibitem[Osi16]{MR3430352}
Denis~V. Osin.
\newblock Acylindrically hyperbolic groups.
\newblock {\em Trans. Amer. Math. Soc.}, 368(2):851--888, 2016.

\bibitem[RST]{JDT2018}
Jacob Russell, Davide Spriano, and Hung~Cong Tran.
\newblock Convexity in hierarchically hyperbolic spaces.
\newblock Submitted. arXiv:1809.09303.

\bibitem[Sis]{Sisto}
Alessandro Sisto.
\newblock On metric relative hyperbolicity.
\newblock Preprint. arXiv:1210.8081.

\bibitem[Sis16]{MR3519976}
Alessandro Sisto.
\newblock Quasi-convexity of hyperbolically embedded subgroups.
\newblock {\em Math. Z.}, 283(3-4):649--658, 2016.

\bibitem[Tra]{Tran2017}
Hung~C. Tran.
\newblock On strongly quasiconvex subgroups.
\newblock To appear in \emph{Geom. Topol.}, arXiv:1707.05581.

\bibitem[Tra15]{MR3361149}
Hung~Cong Tran.
\newblock Relative divergence of finitely generated groups.
\newblock {\em Algebr. Geom. Topol.}, 15(3):1717--1769, 2015.

\end{thebibliography}
\end{document}